\documentclass[a4paper,12pt,reqno]{amsart}
\usepackage{amssymb, paralist, xspace, graphicx, url, amscd, euscript, mathrsfs,stmaryrd,epic,eepic,color}
\usepackage[colorlinks]{hyperref}
\usepackage{xypic}
\usepackage[all,cmtip]{xy}
\usepackage{booktabs}
\usepackage{comment}
\usepackage{colortbl}
\usepackage{diagbox}
\usepackage{color}
\usepackage{lipsum}
\usepackage{multirow}
\usepackage{mathtools}
\usepackage{tikz-cd}
\usepackage[margin=1.3in]{geometry}
\usetikzlibrary{arrows,shapes,chains}
\usepackage{enumitem}
\setlist[enumerate]{leftmargin=*}
\title{ Cohomology of moduli space of cubic fourfolds I}
\author{Fei si}
\address{\ Shanghai Research Center For Mathematical Science,
Shanghai, 200433, People's Republic of China}
\email{fsi15@fudan.edu.cn}
\DeclareMathOperator{\Aut}{Aut}
\DeclareMathOperator{\codim}{codim}
\DeclareMathOperator{\stab}{stab}
\DeclareMathOperator{\IH}{IH}
\DeclareMathOperator{\IP}{IP}
\DeclareMathOperator{\diag}{diag}

\numberwithin{equation}{section}
\newcommand{\bB}{{\mathbb B}}
\newcommand{\bZ}{{\mathbb Z}}
\newcommand{\bC}{{\mathbb C}}
\newcommand{\bQ}{{\mathbb Q}}
\newcommand{\bP}{{\mathbb P}}
\newcommand{\bR}{{\mathbb R}}
\newcommand{\bH}{{\mathbb H}}

\newcommand\cO{{\mathcal{O}}}
\newcommand\HH{{\mathcal{H}}}

\newcommand\cM{{\mathcal{M}}}
\newcommand\aL{{\mathcal{L}}}
\newcommand\ft{{\mathfrak{t}}}

\newcommand\cB{\mathcal{B}}

\newcommand\DD{{\mathcal{D}}}

\newcommand\cF{{\mathcal{F}}}

\newcommand{\q}{/\!\!/}
\newcommand\Proj {{\text{\rm Proj}}}

\newtheorem{thm}{Theorem}[section]

\newtheorem{cor}[thm]{Corollary}
\newtheorem{prop}[thm]{Proposition}

\theoremstyle{definition}
\newtheorem{defn}[thm]{Definition}

\newtheorem{exmp}[thm]{Example}

\newtheorem{rem}[thm]{Remark}
\theoremstyle{remark}

\begin{document}

\begin{abstract}
In this paper we compute the cohomology of moduli space of cubic fourfolds with ADE type singularities relying on Kirwan's blowup and Laza's GIT construction. More precisely, we obtain the Betti numbers of Kirwan's resolution of the moduli space. Furthermore, by applying decomposition theorem we obtain the Betti numbers of  the intersection cohomology of Baily-Borel compactification of the moduli space.
\end{abstract}

\maketitle
\pagestyle{myheadings} \markboth{\hfill Fei Si
\hfill}{\hfill Cohomology of moduli space of cubic fourfold\hfill}

\section{Introduction}

The study of cubic fourfolds and their moduli space  is a classical topic in algebraic geometry and has attracted lots of attentions in various aspects. The purpose of  this paper is to investigate the topology of the coarse moduli space $\mathcal{M}$  of cubic fourfolds with ADE singularities at worst and its various compactifications. The celebrated  works of Voisin \cite{MR860684} \cite{voisin2008theoreme}, Hassett\cite{hassett2000special},  Laza \cite{laza2010moduli} and  Looijenga \cite{looijenga2009period} establish the global Torelli theorem for the cubic fourfolds completely. Thus we can identify the moduli space $\cM$ as the complement $\mathcal{D}/\Gamma-\HH_\infty$ of a Heegner divisor $\HH_\infty$ in the  Shimura variety $\mathcal{D}/\Gamma$ (see \ref{torelli}). This provides many compactifications of $\mathcal{M}$ from the arithmetic side, e.g.  Baily-Borel's compactification \cite{MR554917}, Looijenga's \cite{looijenga2003compactifications} and  toriodal compactifications \cite{MR0457437}. These compactifications  imply that $\cM$ is a  quasi-projective variety.  

The cohomology of moduli space is a basic invariant of particular interest for the moduli space and it is a part of cycle theory on the moduli space. The cohomology rings  of moduli spaces $M_g$ of genus $g$ curves have been studied  decades  since the work of Mumford. Mumford defined the tautological cohomological rings of $M_g$ as subrings of the cohomology rings generated by kappa classes and similarly defined  tautological Chow rings of $M_g$.  It is expected the  tautological cohomological rings are isomorphic to tautological Chow rings via cycle classes maps (see  Question 0.1 in \cite{rahul}). But so far it is only known for genus less than 23.  The cohomology ring  of  the moduli spaces $\cF_g$ of quasi-polarised K3 surface of fixed degree $2g-2$ has also  many progress recently (see \cite{MR3639598}, \cite{PY} \cite{MR3953432}). The tautological ring of  moduli space $\cF_g$ involves more  ingredients than that of $M_g$. In \cite{MR3639598} and  \cite{PY}, it is shown that tautological rings are generated by Noether-Lefschetz cycles.  Cubic fourfolds and K3 surfaces share similar Hodge theory and  the Torelli thorem holds for the both. So it is natural to ask the similar questions for moduli spaces of cubic fourfolds. This is the main motivation for  us to study the topology of $\cM$ and its compactifications  as the first step to the similar picture for cubic fourfolds.

For a topological space $Y$, we denote by
$$P_t(Y):= \mathop{\sum} b_i(X)t^i,\ \ \IP_t(Y):= \mathop{\sum} \dim \IH^i(X) t^i$$
 the  Poincar\'e polynomial of singular cohomology and intersection cohomology ( with respect to middle perversity) of $Y$. Let $\widetilde{\mathcal{M}}$ be the partial desingularization of the GIT compactification $\overline{\mathcal{M}}$ in the sense of Kirwan.  The first main result in the paper is
\begin{thm} \label{thm1.1} The Poincar\'e polynomial of $\widetilde{\mathcal{M}}$ is given by 
\begin{equation*}
    \begin{split}
        P_{t}(\widetilde{\mathcal{M}})=&1+9t^{2}+26t^{4}+51t^{6}+81t^{8}+115t^{10}+152t^{12}+193t^{14}\\
        &+236t^{16}+280t^{18}+324t^{20}+280t^{22}+236t^{24}+193t^{26}\\
        &+152t^{28}+115t^{30}+81t^{32}+51t^{34}+26t^{36}+9t^{38}+t^{40}
    \end{split}
\end{equation*}
\end{thm}
Using the explicit resolution of of the  period maps, we also compute the intersection  cohomology of the Baily-Borel compactification of the Shimura variety $\mathcal{D}/ \Gamma$ .
\begin{thm}
The intersection cohomology Poincar\'e polynomial of $\overline{\mathcal{D}/ \Gamma}^{BB}$ is given by
\begin{equation*}
    \begin{split}
    \IP_t(\overline{\mathcal{D}/ \Gamma}^{BB})=&1+2t^2+5t^4+13t^6+24t^8+38t^{10}+54t^{12}+70t^{14}\\
    &+88t^{16}+107t^{18}+137t^{20}+107t^{22}+88t^{24}+70t^{26}\\
    &+54t^{28}+38t^{30}+24t^{32}+13t^{34}+4t^{36}+2t^{38}+t^{40}.
    \end{split}
\end{equation*}
\end{thm}

\begin{rem}
It is interesting  to  note that by  recent results of Liu \cite{2020arXiv200714320L}, the GIT moduli space of cubic fourfolds is isomorphic to the K-moduli space of cubic fourfolds, that is, the space of  isomorphic classes of cubic fourfolds admitting  Kahler-Einstein metrics. So our computations also provide cohomological results on the  K-moduli space  (see  \cite{2020arXiv201110477X} for  a nice survey of K-moduli spaces).  
\end{rem}

\begin{rem}
 We are most interested in  cohomology of the open part $\DD/\Gamma$ and the complement of Heegener divisor  $\DD/\Gamma-\HH_\infty$. But at present there are some technical difficulty. The problem will be investigated in the future. 
\end{rem}
\begin{rem}
As a complement of a divisor of Shimura variety, the cohomology of  $\cM$  is closely related to the representation theory (see \cite{MR554917} ). Our computation here is based on geometric results. It is expected there is a representation theoretic explanation. After the computational results,  it is an interesting topic to study  the generators of the intersection cohomology in each degree and ask whether these  generators are generated by special cycles (for example, see \cite{MR2594629}). 
\end{rem}

\begin{rem}
It is worth to mention that it is still unknown whether the tautological Chow ring of $\cF_g$ and its tautological cohomology ring are isomorphic.
\end{rem}

The strategy of our approach is as follows: first, the equivariant cohomology of GIT quotient space $\overline{\mathcal{M}}$ can be computed by the stratifications. This relies on from Kirwan's general theory on cohomology of quotient space. Then we apply the partial desingularization procedure, that is, take a successive blowups along GIT strictly semistable loci. And then We  keep track of the change of  cohomology  for each blowup in  partial desingularization and thus we can obtain the cohomology of $\widetilde{\cM}$. For the computation of intersection cohomology of $\overline{\mathcal{D}/ \Gamma}^{BB}$, we need to make use of the  geometry of moduli spaces.  The Torelli theorem provides a birational map $p: \overline{\mathcal{M}} \dashrightarrow \overline{\mathcal{D}/ \Gamma}^{BB}$ between GIT compactification $\overline{\mathcal{M}}$ and the Baly-Borel compactification $\overline{\mathcal{D}/ \Gamma}^{BB}$. The birational map can be explicitly resolved via Kirwan's partial desingularization , that is, there is a diagram
\begin{center}
    \begin{tikzcd}[column sep=small] & \arrow[dl]  \widehat{\mathcal{M}}  \arrow[dr] & \\ \overline{\mathcal{M}} 
\arrow[rr, dashrightarrow, "p"] & & \overline{\mathcal{D}/ \Gamma}^{BB}
\end{tikzcd}
\end{center}
where $\widehat{\mathcal{M}} $ is the intermediate space during the Kirwan's partial desingularization. One can use blowup formula of the intersection cohomlogy reversely  to get the cohomology of  $\widehat{\mathcal{M}}$ from  the cohomology of $\widetilde{\cM}$.  Then the  decomposition theorem of Beilinson-Bernstein-Deligne-Gabber \cite{BBDG} will provide a way to compute  the intersection  cohomology of Baily-Borel compactification of the Shimura variety $\mathcal{D}/ \Gamma$. 

This strategy has been worked out for the moduli space of  K3 surfaces of degree 2 (see \cite{kirwan1989cohomology}, \cite{kirwan1989cohomology2}) and the moduli space of cubic threefolds (see \cite{casalaina2019cohomology}). The main difficulty in our case is that the boundary strata of the GIT compactification $\overline{\cM}$ are much more complicated than the above two cases, but the our observation is that the GIT moduli space of K3 surfaces of degree 2 will appear as an exceptional divisor in Kirwan's desingularization of $\overline{\cM}$, then the computations of  Kirwan-Lee in \cite{kirwan1989cohomology2} will help us to simplify the computations.

\subsection*{Outline}

The paper is organised as follows: In section  \ref{se2}, we review the contruction of moduli space $\cM$, its GIT compactification and the  cohomology theory used in this paper. We also introduce Kirwan's partial desingularization package.  In section  \ref{se4}, we use Kirwan's methods to compute the cohomology of the partial resolution  $\widetilde{\mathcal{M}}$. In section \ref{se5}, we introduce the global Toreli theorem for cubic fourfolds and use the decomposition theorems to compute the intersection cohomology of the Baily-Borel compactification $\overline{\mathcal{D}/ \Gamma}^{BB}$ of the  moduli space of cubic 4-folds. 

\subsection*{Notations and Conventions}
\begin{enumerate}
     \item $\mathcal{M}$ the moduli space   of cubic fourfolds with ADE singularities;
     \item $\overline{\mathcal{M}}$ the GIT comaptification of $\mathcal{M}$;
    \item $\widetilde{\mathcal{M}}$ the Kirwan's desingularization space;
    \item $\mathbb{C}[x_{0},x_{2},...,x_{n}]_{d}$ means degree $d$ homogeneous polynomials in $n+1$ variables;
    \item $l(x),q(x),c(x)$ means linear, quadratic and cubic forms in $x$;
    \item $\{ \hbox{polynomial} \}$ means the vector space spanned by the monomials of the polynomial;
    \item $\alpha,\mu,\gamma,\delta,\cdots$ the  strata of GIT boundaries;
    \item $Z_\alpha,Z_\mu,Z_\gamma\cdots$ the parametrizing space of strata $\alpha,\mu,\gamma,\delta,\cdots$;
    \item $R_\alpha,R_\mu,R_\gamma\cdots$ the stabilizer subgroup of strata $\alpha,\mu,\gamma,\delta,\cdots$;
    
    \item $E_\alpha,E_\mu,E_\gamma\cdots$ the exceptional divisor of Kirwan blowups;
    \item $N(R)$ the normalier subgroup of a subgroup $R$ in group $G$;
    \item $T^n$ the $n$-dimensional complex torus;
    
    \item $BG$ the classifying space of a group $G$;
    \item $\stab(\beta)$ the stabilier subgroup of a vector $\beta$ in Lie algebra by adjoint action;
    \item All cohomology theory $H^\ast,\ \IH^\ast,\cdots $ will take $\bQ$-coefficients;
    \item $D_c(X)$ the bounded derived category of constructible sheaves with $\bQ$-coefficients;
    \item $Sym^n(X)$ the $n$-th symmetric product of a variety $X$.
\end{enumerate}

\section{Preliminaries} \label{se2}
\subsection{Moduli space of Cubic fourfolds}
We work over $\bC$. A cubic fourfold $X$ is a hypersurface in $\bP^5$ defined by a homogeneous polynomial of degree $3$.  
\begin{defn}
We call a cubic fourfold $X$ has ADE singularities if it has only isolated singularities and each singularity germ is a 4-dimensional  hypersurface singularity in $(\bC^5,0)$ that can be written as $x_4^2+x_5^2+f(x_1,x_2,x_3)$ where $f(x_1,x_2,x_3)$  is the equation of the surface singularity  of ADE type.
\end{defn}

Let $\cM$ be the coarse moduli space parametrizing isomorphic classes of cubic fourfolds with ADE singularities. It can be constructed  as follows: The Hilbert scheme of cubic hypersurfaces in $\bP^5$ is isomorphic to the projective space  $\bP(\mathbb{C}[x_{0},x_{2},...,x_{5}]_3)$. The action $G=SL(6,\bC)$ on  $\bP^5$ will induce the action $G$ on $\mathbb{C}[x_{0},x_{2},...,x_{5}]_3\cong H^0(\bP^5,\cO(3))$ and so on  $\bP(\mathbb{C}[x_{0},x_{2},...,x_{5}]_3)$. Let $U \subset \bP(\mathbb{C}[x_{0},x_{2},...,x_{5}]_3)$ be the locus of cubic fourfolds with ADE singularities and it is a open subset. By choosing any ample $G$-linearization  $L$ on $\bP \mathbb{C}[x_{0},x_{2},...,x_{5}]_3$, we have a GIT stability for the point  in $\bP \mathbb{C}[x_{0},x_{2},...,x_{5}]_3$ in the sense of Mumford \cite{mumford1994geometric}. We say a cubic fourfold $X$ is (semi) stable if its associated point $[X]\in \bP (\mathbb{C}[x_{0},x_{2},...,x_{5}]_3)$ is GIT (semi)stable. By a result of Laza, we know
\begin{thm}[\cite{laza2009moduli}, Theorem 5.6]
A cubic fourfold with ADE singularities  is stable.
\end{thm}
As the stabilier group of a stable point is finite, the theorem implies the quotient stack $[U/G]$ is a Deligne-Mumford stack and $\cM \cong U/G$.  Clearly, $\cM$ is not a compact space. The GIT theory provide a natural compactification  $$\overline{\cM}:=\bP(\mathbb{C}[x_{0},x_{2},...,x_{5}]_3) \q G =\Proj(R(\bP(\mathbb{C}[x_{0},x_{2},...,x_{5}]_3),L)^G )$$of $\cM$ also by the theorem where $R(\bP(\mathbb{C}[x_{0},x_{2},...,x_{5}]_3),L)^G$ is the $G$-invariant section rings. Let $\cM^s \subset \overline{\cM}$ be the locus of GIT stable points. we call $\overline{\cM}-\cM^s$ the GIT boundary strata, which parametrize the minimal orbits of semistable cubic fourfolds. For the purpose of computation, we need to understand explicit geometry of these boundary strata. Let us recall Laza's analysis of GIT stability.
\begin{prop} [ Prop 2.6 \cite{laza2009moduli}]
  A strictly semi-stable cubic fourfold with minimal orbit have defining  equation of  the following type:
\begin{itemize}
  \item $\alpha$:  $x_{0}q_{1}(x_{2},..,x_{5})+x_{1}q_{2}(x_{2},..,x_{5})=0;$
  \item $\mu$:   $ax_{0}x_{4}^{2}+x_{0}x_{5}l_{1}(x_{2},x_{3})+bx_{1}^{2}x_{5}+x_{1}x_{4}l_{2}(x_{2},x_{3})+c(x_{2},x_{3})=0;$
  \item $\gamma$:  $x_{0}q(x_{3},..,x_{5})+x_{1}^{2}l_{1}(x_{3},...,x_{5})-2x_{1}x_{2}l_{2}(x_{3},...,x_{5})+x_{2}^{2}l_{3}(x_{3},...,x_{5})=0;$
  \item $\delta$:  $x_{0}q(x_{4},x_{5})+f(x_{1},x_{2},x_{3})=0.$
\end{itemize}
where l, q, f means linear, quadratic and cubic equations respectively. Thus, we have the following stratification $$\overline{\mathcal{M}}-\mathcal{M}^s=\alpha \cup \mu \cup \gamma \cup \delta .$$
\end{prop}
We still use the notation $\alpha,\delta, \cdots \mu$ to denote the boundary  strata corresponding the equation of type  $\alpha,\delta, \cdots \mu$. So we have $$\overline{\cM}-\cM=\alpha \cup \cdots \cup \mu. $$ 
The incidence relation is given by the  figure \ref{gitboundary2}
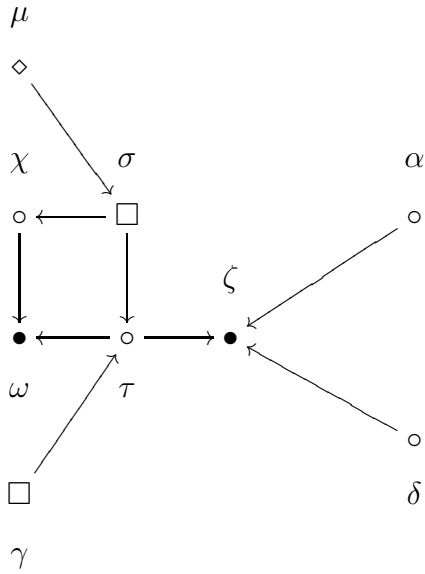
\begin{figure}[htb]
$$\xymatrix@R=.25cm{
&{\mu}   &&                  &&       \\
&{\diamond}\ar@{->}[dddr]&&          &&          \\
\\
&{\chi}&{\sigma}   &&                      &{\alpha}             \\
&{\circ}\ar@{->}[dd]  \ar@{<-}[r]&{\Box}\ar@{->}[dd]&                      &&{\circ}\ar@{->}[ddll]                \\
 &&                      &{\zeta}\\
&{\bullet}\ar@{<-}[r]&{\circ}\ar@{->}[r]                  &{\bullet}        \\
&{\omega}  &{\tau}  &                      &&                    \\
& && && {\circ}\ar@{->}[uull] \\
&{\Box}\ar@{->}[uuur]&&                       && {\delta}                     \\
&{\gamma}   &&                     &&                 \\
}
$$
\caption{ The incidence of boundary components of $\mathcal{M}$ in $\overline{\mathcal{M}}$ } \label{gitboundary2}
\end{figure}
where  $\tau $ is a curve that parametrizes cubic fourfolds with equation of the form  $$ \det\left(
 \begin{array}{ccc}
x_{0} & x_{1} & ax_{2} \\
x_{1} & x_{5} & x_{3} \\
x_{2} & x_{3} & x_{4} \\
\end{array}
\right)=0$$
 and $\zeta$ is the point representing $x_{0}x_{4}x_{5}+x_{1}x_{2}x_{3}$,
     $\omega$ is the  point representing  $$-\det \left(
              \begin{array}{ccc}
x_{0} & x_{1} & x_{2} \\
 x_{1} & x_{5} & x_{3}\\         
 x_{2} & x_{3} & x_{4} \\
 \end{array}
\right)=0
    . $$
By summary of results in \cite{laza2009moduli}, we also get the  stabilizers  of the boundary strata.
\begin{thm}[\cite{laza2009moduli}]
The stabilizers of general points  the GIT boundary strata are one of the following (up to a conjugate)
 \begin{itemize}
   \item $R_{\omega}\cong SL(3,\mathbb{C})$
   \item $R_{\zeta}\cong (\mathbb{C}^{\ast})^{4}$
   \item $R_{\chi}\cong SL(2,\mathbb{C})$
   \item $R_{\delta}\cong SO(2)(\mathbb{C})\times \{\diag(t^{-2},t,t,1,1,1):t\in \mathbb{C}^{\ast}\}$
   \item $R_{\tau}\cong \{\diag(t^{2},t,1,t^{-1},t^{-2},1):t\in \mathbb{C}^{\ast}\}\times \{\diag(t^{4},t,t,t^{-2},t^{-2},t^{-2}):t\in \mathbb{C}^{\ast}\} $
   \item $R_{\alpha}\cong \{\diag(t^{2},t^{2},t^{-1},t^{-1},t^{-1},t^{-1}):t\in \mathbb{C}^{\ast}\} $
   \item $R_{\gamma}\cong \{\diag(t^{4},t,t,t^{-2},t^{-2},t^{-2}):t\in \mathbb{C}^{\ast}\} $
   \item $R_{\mu}\cong \{\diag(t^{2},t,1,1,t^{-1},t^{-2}):t\in \mathbb{C}^{\ast}\} $
 \end{itemize}
 where $\chi$ is a curve that parametrizes cubic fourfold with equations of the form 
 $$ bx_{5}^{3}+\det\left(
                                                                        \begin{array}{ccc}
                                                                          x_{0} & x_{1} & x_{2}+2ax_{5} \\
                                                                          x_{1} & x_{2}-ax_{5} & x_{3} \\
                                                                          x_{2}+2ax_{5} & x_{3} & x_{4} \\
                                                                        \end{array}
                                                                      \right)=0.
 $$
Here $a,b\in \bC$ and the inclusion relation is given by the figure \ref{subgroup}.
 \end{thm}
 
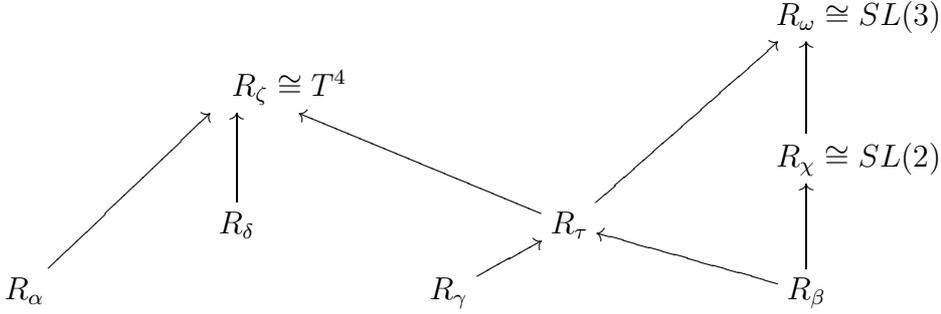
\begin{figure}[htb] \label{subgroup}
\begin{center}
$$
\xymatrix@R=.25cm{
&&&&{\ \ \ \ \ \ \ \ \ \ R_{\omega}\cong SL(3)}\ar@{<-}[dd] \ar@{<-}[dddl]&\\
&{\ \ \ \ \ \ \ \ \ \ R_{\zeta}\cong T^4}\ar@{<-}[dd] \ar@{<-}[dddl]\ar@{<-}[ddrr]\\
&&&&{\ \ \ \ \ \ \ \ \ \ R_{\chi}\cong  SL(2)}\ar@{<-}[dd] \\
&{R_{\delta}}&&{R_{\tau}}\ar@{<-}[dl]\ar@{<-}[dr]&\\
{R_{\alpha}}&&{R_{\gamma}}&&{R_{\beta}}
}
$$
\end{center}
\caption{The stabilizers of general points in the boundary strata}\label{listr}
\end{figure}
Based on the above results, we have explicit results of boundary strata.
\begin{prop} \label{3.3}
The GIT boundaries $\overline{\mathcal{M}}-\mathcal{M}$ are the following strata:
\begin{enumerate}
  \item 1-dimensional strata: $\alpha \cong \mathbb{P}^{1},\  \delta \cong \mathbb{P}^{1}, \tau \cong \mathbb{P}^{1},\  \chi \cong \mathbb{P}^{1}$.
  \item 2-dimensional strata: $\gamma $ is $\mathbb{P}^{1} \times \bC \  $,    $\ \phi \cong \mathbb{P}^{1} \times \mathbb{P}^{1} $.
  \item 3-dimensional strata: $ \mu \cong \mathbb{P}(1,3,6,8)$,  $\varepsilon \cong \mathbb{P}^{1}\times \mathbb{P}(1,2,3)  $.
\end{enumerate}
\end{prop}
\begin{proof}
$\chi \cong \mathbb{P}^{1}$ is shown in section 4.1 in  \cite{laza2010moduli}.  
By Lemma 4.5 in \cite{laza2009moduli}, we can write the defining equation of $\delta$ as $$x_{0}x_{4}x_{5}+f(x_{1},x_{2},x_{3})$$ and it's semi-stable iff the cubic $f(x_{1},x_{2},x_{3})$ has node at worst.  By Luna's slice theorem, $\delta$ is isomorphic to GIT quotient of plane cubics and thus
$$\delta \cong \mathbb{P}|\mathcal{O}_{\mathbb{P}^{2}}(3)|\q  SL(3) \cong \mathbb{P}^{1}.$$
For each element, say $F=x_0q_1+x_1q_2$ in $\alpha$, it can be viewed as a pencil of quadratics in $\mathbb{P}^3$. So by example 6.18 in  \cite{mukai2003introduction}, we have  $$\alpha\cong \Proj((Sym^4(\mathbb{C}))^{SL_2})=\Proj(\mathbb{C}[g_2,g_3])\cong \mathbb{P}^1.$$
For $\gamma$, note that by Lemma 4.6 in \cite{laza2009moduli}, a cubic fourfold in locus $\gamma$ has a normal form $$x_0(x_3^2-x_4x_5)+x_1^2x_4-2x_1x_2l(x_3,x_4,x_5)+x_2^2x_5 $$ and its stablizer group $\mathbb{C}^\ast=\{\ \diag(1,t^{-1},t,1,t^2,t^{-2}):\ t\in \bC^\ast \ \}$. By Luna's slice theorem, this implies that $\gamma$ is isomorphic to the quotient of $\mathbb{C}^\ast$ action on $\mathbb{A}^3(a,b,c)$  where $a,b,c$ are coefficients in the normal form of $l(x_3,x_4,x_5)=ax_3+bx_4+cx_5$,  then $$\gamma\cong \mathbb{C} \times \mathbb{P}^1.$$
For $\beta$, we first use the action $GL(2)$ on $x_2,x_3$ to reduce the problem to consider the torus action $$\diag\{\ (t^{a_0},t^{a_1},t^{a_2},t^{a_3},t^{a_4},t^{a_5}):\ a_2=a_3,\ \sum a_i=0\ \}$$ on the space $$\bP V =\bP^6(y_0,\cdots,y_6)$$  where $V$ is the vector space spanned by the monomials $$\{x_1x_3x_4,\ x_0x_4^2, x_5x_1^2,\ xx_2^3,\ x_2x_3^2, \ x_3x_2^2, x_3^3\}.$$
then we compute its invariant ring which has the  minimal generators $$y_0,\ y_1y_2y_5,\ (y_1y_2)^2y_4y_6, \ (y_1y_2)^3y_3y_6$$
and thus we obtain $\beta \cong \mathbb{P}(1,3,6,8)$. From \cite{laza2009moduli} we know that the loci $\varepsilon $  parametrizes  the cubic 4-folds singular along an irreducible rational normal curve of degree 4. By Proposition 6.6 in \cite{laza2009moduli}, the parameter space for such 4-folds is the product of $Sym^{4}(\mathbb{P}^{1})\cong \bP^4$ and affine space $\mathbb{A}(a,b,c)$ with natural action $SL(2)$ on $\mathbb{P}^{1}$ and $\mathbb{C}^{\ast}$ on $x_{5}$ in the equation above, thus $$\varepsilon \cong  \mathbb{P}^{4}\q SL(2) \times \mathbb{P}(1,2,3) \cong \mathbb{P}^{1}\times \mathbb{P}(1,2,3) .$$
\end{proof}

\subsection{Cohomology Theory} Assume $X$  is  a  smooth variety over $\bC$ with a reductive group $G$ acting on it. We are most interested in its  $G$-equivariant cohomology ( here we use singular cohomology) which captures the $G$-group action. The $i$-th equivariant cohomology $ H^i_G(X,\mathbb{Q})$ is defined by the ordinary cohomology $H^i(EG \times_G X,\mathbb{Q})$ of the topological space  $EG \times_G X$ where $EG \rightarrow BG$ is the universal principal homogeneous $G$ -space and  $EG \times_G X$ is quotient space of $EG \times X$ under action $g(h,x):=(h\cdot g^{-1},g\cdot x)$. We denote the $G$-equivariant Poincar\'e polynomials for a variety $G$ by  $$P_t^G(X):=\mathop{\sum} \limits_{i \geq 0} \dim_\bQ H_G^i(X,\bQ) \ \cdot t^i  .$$The  properties below will be frequently used in this paper  and  one may refer to \cite{hsiang2012cohomology} for more details on  general  equivariant cohomology theory.

\begin{thm}\label{eq}
With notation as above, we have
\begin{enumerate}
  \item
If the quotient $X/G$ has only quotient singularities, then $$H^i_G(X,\mathbb{Q})=H^i(X/G,\mathbb{Q})$$
  \item If the quotient space $X/G$ is contractible, then $$H^i_G(X)=H^i(BG)$$
  \item Let $F\hookrightarrow X \rightarrow B$ be a $G$-equivariant topological fibration on $X$ over the base space $B$ and $F$ is the fiber,  then we have the spectral sequence:
\begin{equation*}\label{}
  H_{G}^{p}(B,H^{q}(F,\mathbb{Q}))\Rightarrow H_{G}^{p+q}(X,\mathbb{Q}).
\end{equation*}
In particular, the spectral sequence implies
\begin{equation}\label{sps}
  P_{t}^{G}(X)=P_{t}^{G}(B)\cdot P_{t}(F).
\end{equation}
\end{enumerate}
\end{thm}
If $X$ is a singular space, we are interested in its intersection which behaves better than the usual 
cohomology. Let $IC_X\in D_c(X)$ be an intersection complex defined by   Goresky-MacPherson  \cite{IH2} , then the intersection cohomology of $X$ is defined to be the hypercohomology $$\IH^i(X,\bQ):=\bH^i(X,IC_X).$$  We refer to \cite{kirwanbook} for the definition of intersection complex and more  details on the theory of intersection cohomology. We have the blowup formula that is obtained by Kirwan.
\begin{prop}
\cite[Proposition 6.2]{kirwan1989cohomology}
Let $Z \subset X$ be a smooth $G$-subvariety with the reductive stablizer subgroup $R$. Let $\widetilde{X}:=Bl_Z(X) \rightarrow X$ be the blow up of $X$ along $Z$, then 
\begin{equation}\label{ic}
\begin{split}
    \dim \IH^i(X/G)=\dim \IH^i(\widetilde{X}/G)\\
    -\mathop{\sum} \limits_{p+q=i} \dim (H^p(Z \q N_0))\otimes H^{\lambda(q)}(\bP /R))^{\pi_0(N)}
\end{split}
    \end{equation}
where  $\lambda(q):=\begin{cases}
 q-2 & \text{if } q\leq \dim \bP /R,\\
q & \text{if } others. 
\end{cases}$
and  $\bP$ is the projection of a normal vector space of any point in $Z$, $N_0$ is the identity component of the normalizer subgroup $N$ of $R$ and $\pi_0(N):=N/N_0$. The actions of $\pi_0(N)$ on $H^p(Z \q N_0)$ and $H^{\lambda(q)}(\bP /R)$ are induced by the actions of $N$ on  $Z^{ss}$ and projective  normal bundle $\bP N_{Z/X} \rightarrow Z$ respectively.
\end{prop} 

\subsection{Kirwan's desingularization package} \label{subse2.}
 Denote by $Z_{R}\subset X$ the locus  whose stabilizer  is $R$. Suppose there are only finitely many locus $$\{\ Z_{R_1},...,Z_{R_r}:\ \dim R_1\geq...\geq \dim R_r \ \}$$ such that all the stabilizers $R_i$ are reductive subgroups of $G$ and all $Z_R$ are smooth, then Kirwan took the blowups successively  along these locus (see \cite{kirwan1985partial},\cite{kirwan1986rational})
$$ \widetilde{X}=Bl_{ \widetilde{Z_{R_r}} } \rightarrow  \cdots \rightarrow Bl_{Z_{R_1}}X\rightarrow X$$
where $\widetilde{Z_{R_r}}$ is the strict transformations of  $Z_{R_r}$ and showed the $G$-action can be lifted to  $ \widetilde{X}$ under suitable polarization. Moreover, it commutes with the GIT quotient
\begin{equation}
    \begin{tikzcd}
       \widetilde{X} \arrow[d] \arrow[r]&  X \arrow[d]\\
     \widetilde{X} \q G \arrow[r] &   X\q G
    \end{tikzcd}
    \end{equation}
In this way, after finite steps, Kirwan  obtained a partial resolution $\widetilde{X} \q G$ of $X\q G$, which has at worst only orbifold points. We call the final blowup space the Kirwan's  desingualrization space of $X\q G$. To study the cohomology of $\widetilde{X} \q G$, Kirwan developed  several  useful cohomological formulas: 

\begin{enumerate}
    \item Cohomology formula for GIT quotient: \label{quo}
let $L$ be a  $G$ linearlised polarization. We can choose a  $G$-equivalent embedding $X\hookrightarrow \bP^N$ via $L$. Let  $T \subset G$ be a maximal torus  and $\ft $ be its Lie algebra, fix a positive Weyl chamber $\ft_{+} \subset \ft$, then define  the index set $\cB_0$ consists of  $\beta\in \ft_{+}$ such that $\beta$ is the closest point  to the origin $0$ of  the nonempty convex hull $con(\alpha_{1},...,\alpha_{m})$ generated by some weights $\alpha_{1},...,\alpha_{m}$. We call a vector in $\cB_0$ index vector. Fix a a norm $|\cdot |$ on $\ft$ (e.g, the one induced by killing form), set 
$$Z_\beta=\{[x_0,\cdots,x_N]\in X :\  x_i=0, \hbox{if}\  \alpha_i.\beta \neq |\beta|^2 \}$$
$$Y_\beta=\{[x_0,\cdots,x_N]\in X :\  x_j=0, \hbox{if}\  \alpha_j.\beta=|\beta|^2 \ \& \ \exists  \ \alpha_i.\beta \neq |\beta|^2 \},$$
then there is a natural retraction map $$p_\beta:\ Y_\beta \rightarrow Z_\beta .$$
Denote by $X^{ss}$ the semistable locus of $X$ with respect to the polarization $L$ in the sense of  Mumford's GIT and $Z_\beta^{ss}$ the locus of semistable points in $Z_\beta$ and let $$Y_\beta^{ss}:=p_\beta^{-1}(Z_\beta^{ss}),\ \ S_\beta:=G\cdot Y_\beta^{ss},$$
then combing  the theory of moment maps and relations of symplectic reduction and geometric invariant theory, it is shown in \cite{kirwan1984cohomology} that $\{S_\beta \}_{\beta\in \cB_0}$ gives $X$ a $G$-equivariant perfect Morse stratification .  In particular, for $\beta=0$, $S_0=X^{ss}$. Using such stratification, Kirwan obtained the following formula of  Poincar\'e 's polynomials, 
\begin{equation} \label{pp}
  P^{G}_{t}(X^{ss})=P_{t}(X)P_{t}(BG)- \mathop{\sum }  \limits_{0\neq \beta \in \cB_0} t^{2\codim (S_{\beta})}P^{\stab(\beta)}_{t}(Z^{ss}_{\beta}).
\end{equation}
We will call the term $\mathop{\sum }  \limits_{0\neq \beta \in \cB_0} t^{2\codim (S_{\beta})}P^{\stab(\beta)}_{t}(Z^{ss}_{\beta})$ the \textbf{removing part} in the formula \ref{pp}.
Moreover, there is a natural identification $$S_\beta \cong G \times_{P_\beta}Y_\beta^{ss}$$
where $P_\beta \leq G$ is the parabolic subgroup associated to $\beta$. In this way,  we have also have a dimension formula 
\begin{equation}\label{dim}
    \dim S_\beta= \dim G + \dim Y_\beta^{ss} - \dim P_\beta.
\end{equation}
\item Cohomology formula for blowups: assume $R$ is a reductive subgroup of $G$ with respect to the locus $Z_{R}$. We take the blow up $$\pi : \  \widetilde{X}\rightarrow X^{ss}$$ along the smooth center $G\cdot Z_{R}^{ss}$. Let $N(R)$ be the normalizer subgroup of $R$ in $G$ and $d_R$ be the complex codimension of $Z_R$ in $X$, then
 Kirwan's blowup formula \cite{kirwan1989cohomology} gives
\begin{equation}\label{3.3}
 \begin{aligned}
 P^{G}_{t}(\widetilde{X}^{ss})=P^{G}_{t}(X^{ss})+(t^{2}+...+t^{2d_{R}})P^{N(R)}_{t}(Z_{R}^{ss}) \\
-\mathop{\sum} \limits_{\beta \in \cB_{0,\rho}} t^{2\codim (S_{\beta})}P^{\stab(\beta)\cap N(R)}_{t}(Z^{ss}_{\beta,\rho}).
 \end{aligned}
\end{equation}
Here $\cB_{0,\rho}$ is the index set obtained as in (\ref{quo}) with respect to  the normal representation
\[
\rho:R \rightarrow \Aut(\mathbb{P}\mathcal{N}),
\]
where $\mathcal{N}$ is  the normal vector space  of a point in $Z_{R}^{ss}$ and $Z^{ss}_{\beta,\rho} \subset \bP\mathcal{N} $ is the associated semi-stable strata given by weight $\beta$.

Suppose there are series of blowups $$  \cdots \rightarrow X_2 \rightarrow X_1 \rightarrow  X_0=X.$$ We write the \textbf{correction terms} contributed in the formula of $i$-th blowup as follows:
\begin{equation}\label{3.4}
 \begin{aligned}
 A_{i}(t)\coloneqq & (t^{2}+...+t^{2d_{R}})P^{N(R_{\omega})}_{t}(Z_{R}^{ss}) \\
 &-\mathop{\sum} \limits_{\beta \in \cB_{0,\rho}} t^{2\codim (S_{\beta})}P^{\stab(\beta)\cap N(R_{\omega})}_{t}(Z^{ss}_{\beta,\rho}).
 \end{aligned}
\end{equation}
\end{enumerate}


\section{Cohomology of partial desingularization $\widetilde{\mathcal{M}}$} \label{se4}
In this section, we follow the discussion in section \ref{se2} to do the computation. Let  $\widetilde{\mathcal{M}}$  be the Kirwan's desingularization space of $\overline{\cM}$. Note that $\widetilde{\mathcal{M}}$ has at worst orbifold points and thus the  equivariant  cohomology of $H^\ast_G(\widetilde{\mathcal{M}})$ equals to the usual cohomology $H^\ast(\widetilde{\mathcal{M}})$ by Theorem \ref{eq}. By duality of intersection cohomology,  we will only consider the term in $P^G_{t}(\widetilde{\mathcal{M}})$ of degree lower than 20. Throughout the section, $X=\bP^{55}$ and $G=SL(6,\bC)$.
\subsection{Computations for blowups}
\subsubsection{Computation of  $P^{G}_{t}(X^{ss})$}
According to the formula \ref{pp}, we need to  describe the index set $\cB_0$. It consists of all closest points $\beta$ lying in a positive Weyl chamber $\tau_{+}$ to origin $0$ in  convex hull $con(\alpha_{1},...,\alpha_{5})$ generated by some weights $\alpha_{1},...,\alpha_{5}$.
Let $T$ be a maximal torus in $SU(6,\mathbb{C})$ and its Lie algebra is $$\mathfrak{t}=\bR^5\cong\{\diag(\sqrt{-1}\theta_{0},\sqrt{-1}\theta_{1},...,\sqrt{-1}\theta_{5}): \sum \theta_{i}=0\} .$$
Each 1-parameter subgroup is of the form $$\lambda(t)=\{\diag(t^{r_{0}},t^{r_{1}},t^{r_{2}},t^{r_{3}},t^{r_{4}},t^{r_{5}}):  \sum r_{i}=0\}$$ which   can be  identified as a vector in $\mathfrak{t}$.  For each monomial $x_0^{i_0}\cdots x_5^{i_5}$ of degree $3$, its weight with respect to a  1-parameter subgroup is 
\begin{equation*}
    i_0\theta_0+i_1\theta_1+\cdots+i_5\theta_5=(i_{1}+i-3)\cdot \theta_1+\cdots (i_{5}+i-3)\cdot \theta_5
\end{equation*}
where $i:=i_{1}+i_{2}+i_{3}+i_{4}+i_{5}$ and $(\theta_0,\cdots,\theta_5)\in \mathfrak{t}$ is the vector of 1-parameter subgroup.
So we need to identify the momonials with weight vector  $$W=\{(i_{1}+i-3,i_{2}+i-3,i_{3}+i-3,i_{4}+i-3,i_{5}+i-3):0\leq \sum i_{j}\leq3\}.$$  By choosing a positive root system $$\Phi_{+}=\{(1,-1,0,0,0),(0,1,-1,0,0),$$ $$(0,0,1,-1,0),(0,0,0,1,-1),(0,0,0,0,0,3)\}, $$ we obtain the positive Weyl chamber in $\mathfrak{t}$ as follows $$\ft_{+}=(\theta_{1},..,\theta_{5}):\theta_{1}\geq...\geq \theta_{5}\geq 0\}.$$
Given  weight vectors $\alpha_1,\cdots,\alpha_5\in W$, we can find the closest points $\beta$ from origin $0$ to convex hull $con(\alpha_{1},...,\alpha_{5})$. If $\beta \in \mathfrak{t}_+$, we keep the data $\beta$. If not, we continue to use other weight vectors. With the help of computer,  we find the only data  with codimension $<10$  is the case $\beta=(\frac{3}{5},\frac{3}{5},\frac{3}{5},\frac{3}{5},\frac{3}{5})\in \tau_{+}$. In order to get $P^{G}_{t}(X^{ss})$, we need to compute the removing part. In this case, we have
   \begin{equation*}
       \begin{split}
           \stab(\beta)=&\{\ \left(
                              \begin{array}{cc}
                                a & 0 \\
                                0 & A\\
                              \end{array}
                            \right)
                            \in SL(6,\mathbb{C}):a\cdot \det A=1,A\in GL(5,\mathbb{C})\}\\
            \cong &\mathbb{C}^{\ast}\times SL(5,\mathbb{C})
       \end{split}
   \end{equation*}
where the first factor acts trivially on $Z_{\beta}=\mathbb{P}( \mathbb{C}[x_{1},x_{2},x_{3},x_{4},x_{5}]_{3})$. Thus,
\begin{equation*}
    \begin{split}
        P^{\stab(\beta)}_{t}(Z^{ss}_{\beta})=&\frac{1}{(1-t^{2})}\cdot P^{SL(5,\mathbb{C})}_{t}(Z^{ss}_{\beta})\\
        =&\frac{1}{(1-t^{2})}\cdot P^{SL(5,\mathbb{C})}_{t}(\mathbb{P} \left( \mathbb{C}[x_{1},x_{2},x_{3},x_{4},x_{5}]_{3})^{ss} \right).
    \end{split}
\end{equation*}

\begin{center}
\begin{table}[tbp]
\centering  
\begin{tabular}{|l|c|c|c|c}  
\hline
$\beta$ & Monomials & stable($\beta$)  \\ \hline  %
$1.\ $(0.25,0.25,0.25,0.25) & $\mathbb{C}[x_{1},x_{2},x_{3},x_{4}]_{3}$ &$\left(
                            \begin{array}{cc}
                              a & 0 \\
                             0 & A \\
                            \end{array}
                            \right)$
 \\     \hline      
2. (0.9,0.8,0.7,0.6)&$\{x_{1}x_{2}^{2},x_{1}x_{3}^{2},x_{1}x_{2}x_{4}\}$ &$T^{4}$ \\    \hline     
3. (1,0,0,0) &$x_{1}\cdot\mathbb{C}[x_{2},x_{3},x_{4}]_{2}\oplus x_{0}x_{1}^{2} $&$\left(
\begin{array}{ccc}
               a & 0 & 0 \\                         
               0 & b & 0 \\                         
               0 & 0 & A \\
            \end{array}
                    \right)
$\\  \hline

4.\ (1,1,0.5, 0.5)&$\{x_{1}^{2},x_{2}^{2},x_{1}x_{2}\}\cdot \{x_{3},x_{4}\}$&$\left(                                               \begin{array}{ccc}                               a & 0 & 0 \\
0 & A & 0 \\
0 & 0 & B \\
\end{array}
\right)
$
\\   \hline
5.\ (1.08,0.84,0.72,0.36)&$\{x_{2}^{3},x_{1}x_{3}^{2},x_{4}x_{1}^{2}\}$&$T^{4}$\\   \hline
6.\ (1,1,1,0)&$\mathbb{C}[x_{1},x_{2},x_{3}]_{3}$&$SL(3,\mathbb{C})\times T^{2} $\\   \hline
7.\ (1.2,0.9,0.6,0.3)&$\{x_{2}^{3},x_{1}x_{2}x_{3},x_{4}x_{1}^{2}\}$&$T^{4}$\\  \hline
8.\ (1.25,0.75,0.75,0.25)&$\{x_{1}^{2}x_{4},x_{1}x_{2}x_{3},x_{1}x_{2}^{2},x_{1}x_{3}^{2}\}$&$\left(
                    \begin{array}{cccc}
                    a & 0& 0 & 0 \\
                    0 & b & 0& 0  \\
                    0 & 0 & c & 0\\
                    0& 0 & 0 & A \\
                    \end{array}
                    \right)
$\\
\hline
\end{tabular}
\caption{Unstable stratification.} \label{usstra}
\end{table}
\end{center}
To compute $P^{SL(5,\mathbb{C})}_{t}(\mathbb{P} \left( \mathbb{C}[x_{1},x_{2},x_{3},x_{4},x_{5}]_{3})^{ss} \right)$, we will use Kirwan's formula \ref{pp} once more since these terms can be viewed as the cohomology of GIT quotient space. As before,  we have  unstable data given in the table \ref{usstra}  with help of computer.

In the table \ref{usstra}, by choosing a suitable 1-parameter subgroup, we find some locus  $Z^{ss}_{\beta}=\phi$ and in this case, $P^{\stab(\beta)}_{t}(Z^{ss}_{\beta})=0$. Such index vector is fake. For example, 1-parameter subgroup $$\lambda(t)=\diag(1,t^{3},t^{-1},t^{-1},t^{-1})$$ will destabilize  fourth data. Checking case by case, we have only two nonzero data: 1st and 6th that make contribution.
\begin{enumerate}
    \item For 1st data, $v_{1}=\frac{1}{4}(1,1,1,1)$, by the dimension formula \ref{dim},
the codimension of unstable stratum equal to $35-(24+20-15)=6$.
By the formula \ref{pp}, we have 
$$P^{stab(v_{1})}_{t}(Z^{ss}_{v_{1}})=\frac{1}{(1-t^{2})}\cdot P^{SL(4,\mathbb{C})}_{t}(\mathbb{P}\mathbb{C}[x_{1},x_{2},x_{3},x_{4}]_{3})\ \mod t ^{20}$$
The same method give only two unstable data for $SL(4,\mathbb{C})\curvearrowright \mathbb{P}\mathbb{C}[x_{1},x_{2},x_{3},x_{4}]_{3}$:\\
\begin{enumerate}
    \item $\beta_{1}=(1,1,1)$. Such stratum has codimenion $4$ and
    \begin{equation}\label{1}
\begin{split}
    P^{\stab(\beta_{1})}_{t}(Z^{ss}_{\beta_{1}})&=\frac{1}{(1-t^{2})}\cdot P^{SL(3,\mathbb{C})}_{t}(\mathbb{P}\mathbb{C}[x_{1},x_{2},x_{3}]_{3})\\
    &=\frac{1}{(1-t^{2})} \cdot \frac{1+t^{2}+t^{10}+t^{12}}{(1-t^{4})\cdot(1-t^{6})}
\end{split}
\end{equation}
\item $\beta_{2}=(1,0,0)$. Such stratum has codimenion $5$ and so we have \begin{equation}\label{2}
 P^{\stab(\beta_{2})}_{t}(Z^{ss}_{\beta_{2}})=\frac{1+t^{2}+t^{4}+t^{6}}{(1-t^{2})\cdot(1-t^{4})}-\frac{(1+t^{2})\cdot t^{2}}{(1-t^{2})^{2}}=\frac{1}{1-t^{2}}.
\end{equation}
\end{enumerate}
Combing these data, we have
\begin{equation*}
    \begin{split}
        P^{\stab(v_{1})}_{t}(Z^{ss}_{v_{1}})=&\frac{1}{(1-t^{2})}\cdot\{\frac{1+t^{2}+...+t^{38}}{(1-t^{4})\cdot(1-t^{6})\cdot(1-t^{8})}\\
        &-\frac{t^{12}}{(1-t^{2})}- \frac{t^{8}}{(1-t^{2})} \cdot \frac{1+t^{2}+t^{10}+t^{12}}{(1-t^{4})\cdot(1-t^{6})}\}.
    \end{split}
\end{equation*}
\item For 6th data, the index vector is $v_{2}=(1,1,1,0)$ and thus we have  $$\stab(v_{2})=(\mathbb{C}^{\ast})^{2}\times SL(3,\mathbb{C}).$$ Then the codimension  of the removing strata is  $$34-(24+9-15)=16> 10.$$ Thus by the formula \ref{pp}, we have
\begin{equation*}
    \begin{split}
        P^{\stab(v_{2})}_{t}(Z^{ss}_{v_{2}})=&\frac{1}{(1-t^{2})^{2}}\cdot P^{SL(3,\mathbb{C})}_{t}(\mathbb{P}\mathbb{C}[x_{1},x_{2},x_{3}]_{3}) \ \mod t ^{20}\\
        =& \frac{1}{(1-t^{2})^{2}}\cdot \frac{1+t^{2}+t^{10}+t^{12}}{(1-t^{4})\cdot(1-t^{6})} \ \mod t ^{20}.
    \end{split}
\end{equation*}
\end{enumerate}
Observe the codimension of unstable stratification of this data is given by 
\begin{equation*}
    \begin{split}
        \codim S_{\beta}=&55-(\dim G+\dim Y_\beta-\dim P_\beta)\\
        =&55-(35+34-20)=6.
    \end{split}
\end{equation*}
Here $P_{\beta}$ is a parabolic subgroup consisting of all upper-triangle matrix and thus it has $\dim P_{\beta}=\frac{(6+1)\cdot 6}{2}-1=20$ and $$\dim Y_{\beta}=\# \{\ \alpha \in W:\alpha.\beta\geq \beta.\beta\ \}=34.$$ 
Thus by putting all the discussions above  into the formula \ref{pp}, we have 
\begin{prop}
 \begin{equation}\label{}
 \begin{aligned}
P^{G}_{t}(X^{ss})\equiv & \frac{1-t^{112}}{\mathop{\Pi} \limits_{1\leq i \leq6}(1-t^{2i})}-t^{12}P_{t}^{GL(5)}(\mathbb{P}\mathbb{C}[x_{0},x_{1},x_{2},x_{3},x_{4}]_{3})\\
 \equiv& \frac{1}{\mathop{\Pi} \limits_{1\leq i \leq6}(1-t^{2i})}-\frac{t^{12}}{1-t^2} \frac{1}{\mathop{\Pi} \limits_{1\leq i \leq 5}(1-t^{2i})}\ mod\ t^{20}.
 \end{aligned}
\end{equation}
\end{prop}
Next, we will take the blowups successively along the locus discussed in section \ref{subse2.}.  Here we give the list of locus to be blowuped in the following  table 2.
\begin{table}[tbp]
\centering 

    \begin{tabular}{|c|c|c|}
\toprule  
Blowup locus &  stabilizer group(up to finite index) &  codimension \\
\midrule  
$G \omega$ & $SL(3,\mathbb{C})$ & 27 \\
$G \zeta$ & $T^{4}$   & 24 \\
$G Z_\chi$ &  $SO(2)$   & 21 \\
$G Z_\tau$ & T    & 20 \\
$G Z_\delta$ & $\bC^\ast \times \bC^\ast$   & 19\\
$G  Z_\alpha$ & $\bC^\ast$   &   19 \\
$G  Z_\gamma$ & $\bC^\ast$   &  18 \\
$GZ_{\beta}$ & $\bC^\ast$  &  17\\ 
\bottomrule 
\end{tabular}
\caption{List of datas to be blowed up}\label{usdata}
\end{table}

\subsubsection{Computation of $P^{G}_{t}(X_{1}^{ss})$}
 We take the blow up $$\pi : \  X_{1}\rightarrow X^{ss}$$ along $G\cdot Z_{R_{\omega}}^{ss}$.
By dimension counting, we have $$d_{R}+1=\codim G\cdot Z_{R}^{ss}$$ $$=55-(\dim G-\dim N(R_{\omega}))=55-(35-8)=28.$$
Moreover, we have
\begin{equation*}
    P^{N(R_{\omega})}_{t}(Z_{R}^{ss})=P_{t}(BN(R_{\omega}))=\frac{1}{(1-t^{4})\cdot (1-t^{6})}
\end{equation*}
 since $Z_{R_{\omega}}^{ss}$ is just a point. In \cite[Corollary 4.2]{laza2010moduli},  by using the fact the cubic $4$-fold $\omega$ is the secant variety of Veronese embedding $\mathbb{P}^2\hookrightarrow \mathbb{P}^5$, Laza proved
\begin{prop} \label{laza}
The representation of $R_{\omega}$ on the normal slice $\mathcal{N}_{\omega}$ is isomorphic to $Sym^{6}(\mathbb{C}^{3})$, where $R_{\omega}\cong SL(3,\mathbb{C})$ has the natural representation on $\mathbb{C}^{3}$. In particular, the exceptional divisor $\bP \mathcal{N}_{\omega}  \q R_\omega$ is isomorphic to the GIT quotient space of plane sextic curves.
\end{prop}

A very helpful observation  is the following
\begin{cor} \label{cor} After the 1st blowup, the incidence relations of boundaries on exceptional divisor $\bP\mathcal{N}_{\omega} \q SL(3,\mathbb{C})$  coincide with that of GIT moduli space of degree 2 K3 surfaces, See the figure 3. That is, let $\ \chi_{1},\ \beta_{1},\  \tau_1,\  \beta_1  \subset X_{1}^{ss}/G$ be the strict transformation of  the GIT boundary after the 1st blowup along $G\cdot Z_\omega$,  then
 \begin{equation}\label{}
 \begin{split}
   E_{\omega}^{ss}/G \cap \chi_{1}= pt,\ \  E_{\omega}^{ss}/G \cap \tau_{1}= pt  \\
   E_{\omega}^{ss}/G \cap \gamma_{1}\cong |\mathcal{O}_{\mathbb{P}^{1}}(4)|//SL(2)=\mathbb{P}^1  \\
   E_{\omega}^{ss}/G \cap \beta_{1}\cong \mathbb{P}^{3}// \mathbb{C}^{\ast}=\mathbb{P}(1,2,3) .
   \end{split}
 \end{equation}
\end{cor}
\begin{proof} Under the isomorphism $\bP \mathcal{N}_{\omega}  \q R_\omega \cong \bP Sym^{6}(\mathbb{C}^{3}) \q SL(3)$ in Proposition \ref{laza}, we can identify the stability on both sides.  Since after the first blowup, the rest blowups  restricting to the divisor $\bP \mathcal{N}_{\omega}  \q R_\omega$ are isomorphic to the partial resolution of $\bP \mathcal{N}_{\omega}  \q R_\omega$ in the sense of Kirwan, then the  locus of GIT  strictly semistable locus will coincide with that of the GIT moduli space of plane sextics $ |\mathcal{O}_{\mathbb{P}^{2}}(6)|\q SL(3)$. Such locus have been explicitly described in \cite{kirwan1989cohomology2}. Then  from  the incidence relation in figure \ref{gitboundary2}, we prove the assertion.
\end{proof}

\begin{figure}[tbp]
\centering  
\usetikzlibrary {calc} 
\usetikzlibrary {positioning}
 \begin{tikzpicture} 
 
\draw (0,0) rectangle (5,2); 
 \draw [blue] [thick] (0.2,1) to [edge  label=$\beta_1$](4.8,1) ;
 \draw [red] [thick] (1.5,1.8) to[edge  label=$\gamma_1$] (1.5,0.2);
\node (a) at (1.15,1) {$\chi_1$};
\fill [red] ($(a) + 1/3*(1cm,0)$) circle (2pt);
\node (b) at (3,1) {$\tau_1$};
\fill [blue] ($(b) + 1/3*(1cm,0)$) circle (2pt);
     \end{tikzpicture}
\caption{incidence relation }\label{incidenk3}
\end{figure}
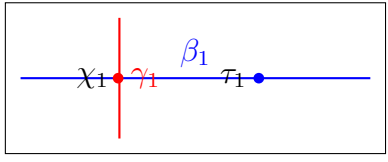

\begin{rem}
  This corollary shows in the 2nd-8th  blowups, the blowup on $ E_{\omega}^{ss}/G \cap \beta_{1}$ will be the same as that in GIT moduli space of  degree $2$ K3 surfaces, then Kirwan-Lee's results (see \cite{kirwan1989cohomology} \cite{kirwan1989cohomology2}) will help us simplify many computations.
\end{rem}

Following the computation in \cite{kirwan1989moduli} for K3 surface and formula \ref{3.4}, we have
\begin{equation}\label{3.5}
 \begin{aligned}
 A_{1}=&\frac{(t^2-t^{56})}{(1-t^{2})\cdot(1-t^{4})\cdot (1-t^{6})} \\
 &-(\frac{t^{50}-t^{56}}{(1-t^{2})\cdot(1-t^{4})(1-t^{6})}+\frac{t^{20}-t^{28}}{(1-t^{2})^3})\\
 \equiv& \frac{t^2}{(1-t^{2})\cdot(1-t^{4})\cdot (1-t^{6})}-\frac{t^{20}}{(1-t^{2})^3}\ \mod\ t^{20}.
 \end{aligned}
\end{equation}

\vspace{0.3cm}

\subsubsection{Computation of $P^{G}_{t}(X^{ss}_{2})$}

Thanks to the disjointness of orbit $G{\omega}$ and $G{\zeta}$, for the second blowup we do not need to consider the effect of the first blowup. So we take the second blowup
\begin{equation}\label{bu2}
\pi : \  X_{2}\rightarrow X_{1}^{ss}
\end{equation}
along $G\cdot Z_{R_{\zeta}}^{ss}$.  It's easy to see that the normalizer of $R_{\zeta}$ in $G=SL(6,\mathbb{C})$  is given by the extension
$$0\rightarrow T^{5} \rightarrow N(R_{\zeta}) \rightarrow S _{6}\rightarrow 0.$$ Here the symmetric group $S_{6}$ is the Weyl group of $R_{\zeta}$. Then
\begin{equation*}
    \begin{split}
        d_{R_{\zeta}}+1=&\codim(G\cdot Z_{R_{\zeta}}^{ss})\\
        =&55-(\dim G-\dim N(R_{\zeta}))=25.
    \end{split}
\end{equation*}
This gives \begin{equation*}\label{} \begin{aligned} (t^{2}+...+t^{2d_{R}})P^{N(R_{\zeta})}_{t}(Z_{R_{\zeta}}^{ss}) & =(t^{2}+...+t^{48})P_{t}(BN(R_{\zeta}))\\ 
&=\frac{(t^{2}-t^{48})}{(1-t^{2})...(1-t^{12})} \end{aligned} \end{equation*} since $G\cdot Z_{R_{\zeta}}^{ss}=G\times_{N(R_{\zeta})}Z_{R_{\zeta}}^{ss}$ and $Z_{R_{\zeta}}^{ss}$ is just a point. Actually, following a lemma of section 4.2 in \cite{kirwan1989cohomology}, we can compute the normal vector space $\mathcal{N}_{\zeta}$ at $\zeta$ as follows
\begin{equation*}
    \begin{split}
        \mathcal{N}_{\zeta}=\mathbb{C}x_{0}^{3}\oplus ...\oplus \mathbb{C}x_{5}^{3}\oplus \{x_{0}^{2},x_{4}^{2},x_{5}^{2}\}\mathbb{C}[x_{1},x_{2},x_{3}]_{1} \\
        \oplus \{x_{1}^{2},x_{2}^{2},x_{3}^{2}\} \mathbb{C}[x_{0},x_{4},x_{5}]_{1}\oplus \mathbb{C}\{x_{0}x_{4}x_{5}+x_{1}x_{2}x_{3}\}.
    \end{split}
\end{equation*}
So the intersections of exceptional divisor $E_2/N(R)$ with the proper transformation of $\alpha, \delta, \tau$ are 3 distinct points. Note that $R_{\zeta}=\{\diag(a,b,c,d,c^{-1}d^{-1},a^{-1}b^{-1}):a,b,c,d\in \mathbb{C}^{\ast}\}$ acts trivially on $x_{0}x_{4}x_{5}+x_{1}x_{2}x_{3}$. By Kirwan, the unstable data is identified with the unstable data of natural action $R_{\zeta}$ on  $\mathbb{P}\mathcal{N}_{\zeta}$. So we only consider this action. Each $1$-parameter subgroup can be written as $\diag\{t^{a_{0}},...,t^{a_{5}}\}$ and the weight is of the form
$$\{\ a\cdot I:\  x^{I}\in \mathcal{N}_{\zeta} \ \}$$ where $x^I=x_0^{i_0}\cdots x_5^{i_5}$ with $i_0+\cdots+i_5=3$ and $a\cdot I=a_oi_o+\cdots a_5i_5$. Note that in the formula \ref{3.4}, the codimension is
\begin{equation*}\label{}
 \begin{aligned}
\codim = 24-\#\{\ a\cdot I\in W:\ a\cdot I>0\}\geq 11 .
\end{aligned}
\end{equation*}
Thus, we obtain
\begin{equation}\label{}
 \begin{aligned}
A_{2}\equiv     \frac{t^{2}}{(1-t^{2})...(1-t^{12})} mod\ t^{20}.
\end{aligned}
\end{equation}

\subsubsection{Computation of $P^{G}_{t}(X^{ss}_{3})$}
   We take the third blowup $$\pi : \  X_{3}\rightarrow X_{2}^{ss}$$ along $G\cdot \widehat{Z}_{R_{\chi}}^{ss}$ where $\widehat{Z}_{R_{\chi}}^{ss}$ is the strict transformation of $Z_{R_{\chi}}^{ss}$ under composition of previous blowups, since $\chi$ contains point $\omega$ and $\zeta$. we have $$\codim(G\cdot \widehat{Z}_{R_{\chi}}^{ss})=55-(\dim G+\dim \widehat{Z}_{R_{\chi}}^{ss}-\dim N(R_{\chi}) )=23.$$
By Proposition \ref{3.3}, we have $$\widehat{Z}_{R_{\chi}}\q N(R_{\chi}) \cong \mathbb{P}(1,3) \cong \mathbb{P}^{1}. $$
This gives
\begin{equation*}\label{}
\begin{aligned}
(t^{2}+...+t^{2d_{R}})\cdot P^{N(R_{\chi})}_{t}(\widehat{Z}_{R_{\chi}}^{ss})
=&\frac{t^{2}-t^{46}}{1-t^{2}}\cdot P_{t}(BR_{\chi})\cdot P^{N(R_{\chi})/R_{\chi}}_{t}(\widehat{Z}_{R_{\chi}}^{ss})\\
=& (\frac{t^{2}-t^{46}}{1-t^{2}})\cdot \frac{1}{1-t^{4}}\cdot (1+t^{2})
\end{aligned}
\end{equation*}
since the action $N(R_{\chi})$ on $\widehat{Z}_{R_{\chi}}$ is isomorphic to the action $N(R_{\chi})/R_\chi$ on $\widehat{Z}_{R_{\chi}}$. In the same paper \cite[Lemma 4.3]{laza2010moduli}, Laza showed the normal representation $\rho:R_{\chi} \curvearrowright \mathcal{N}_{\chi}$ can be identified as  $$ SL(2) \curvearrowright H^{0}(\mathcal{O}_{\mathbb{P}^{1}}(12))\oplus H^{0}(\mathcal{O}_{\mathbb{P}^{1}}(8)). $$
 This gives 
 \begin{equation*}
     \begin{split}
         & \sum t^{2\codim (S_{\beta})}P^{stab(\beta)\cap N(R)}_{t}(Z^{ss}_{\beta,\rho})\\
         =&\frac{t^{24}+t^{26}+t^{28}+t^{30}(1+t^{2})+t^{30}+t^{34}+t^{36}+t^{38}+t^{40}(1+t^{2})}{1-t^{2}} \\
         =&\frac{t^{24}-t^{44}}{(1-t^{2})^{2}}.
     \end{split}
 \end{equation*}
By putting these together, we obtain
\begin{equation}\label{}
\begin{aligned}
A_{3}=& (\frac{t^{2}-t^{46}}{1-t^{2}})\cdot \frac{(1+t^{2})}{1-t^{4}}-\frac{t^{24}-t^{44}}{(1-t^{2})^{2}}\\
\equiv& \ \frac{t^{2}}{(1-t^{2})^2}\ mod\ t^{20}.
\end{aligned}
\end{equation}

\subsubsection{To compute  $P^{G}_{t}(X^{ss}_{4})$}
We take the 4th blowup $\pi : \  X_{4}\rightarrow X_{3}^{ss}$ along $G\cdot \widehat{Z}_{R_{\tau}}^{ss}$. Since  $N(R_\tau)=R_\tau\leq N(R_\omega)=SL(3)$, we have
$$\codim (G\cdot \widehat{Z}_{R_{\tau}}^{ss})=55-(\dim G+\dim Z_{R_\tau}-\dim N(R_\tau))=21. $$
Observe that we can identify the normal representation $R_{\tau} \curvearrowright \mathcal{N}_{\tau} \cong R_{1}\curvearrowright \mathcal{N}_{1}$ where the normal representation $R_{1}\curvearrowright \mathcal{N}_{1}$ the second blowup in \cite[Section 4 ]{kirwan1989cohomology} by corollary \ref{cor}.  Then from \cite[table 2]{kirwan1989cohomology}, we have
\begin{equation}\label{}
\begin{aligned}
A_{4}=&\frac{t^{2}-t^{42}}{1-t^{2}}\cdot P_{t}(B R_{\tau})\cdot (1+t^2) -\frac{t^{18}+t^{20}}{1-t^{2}}\ mod\ t^{20}\\
\equiv & \frac{t^{2}(1-t^2)}{(1-t^{2})^3}-\frac{t^{18}+t^{20}}{1-t^{2}}\ mod\ t^{20}
\end{aligned}
\end{equation}
where the multiplication term $1+t^2$ is due to the geometry of locus $\widetilde{Z_\tau}^{ss}/(N(R_\tau)/R_\tau) \cong \bP^1$ by proposition \ref{3.3}.

\subsubsection{To compute  $P^{G}_{t}(X^{ss}_{5})$} \label{4.1.6}
We take the blowup $\pi : \  X_{5}\rightarrow X_{4}^{ss}$ along $G\cdot \widehat{Z}^{ss}_{R_{\delta}}$.
Note that here $$Z_{R_{\delta}}=\mathbb{P}\{x_{0}q(x_{4},x_{5})+c(x_{1},...,x_{3})\}$$ where $\{x_{0}q(x_{4},x_{5})+c(x_{1},...,x_{3})\}$ means the vector space spanned by the monomials in a general polynomial of the form $x_{0}q(x_{4},x_{5})+c(x_{1},...,x_{3})$. And the normalizer subgroup of such locus is $$N(R_{\delta})=\{\diag(a,A,B):a^{-1}=|A|\cdot|B|, A\in GL(3), B \in GL(2)
 \}.$$ Thus, by the dimension counting, we have $$ \codim  GZ_{R_\delta}=55-(35+12-13)=21.$$  Observe that  $Z_{R_{\delta}}$ contains the point $\zeta
 $ which is represented by the equation $x_{0}x_{4}x_{5}+x_{1}x_{2}x_{3}$, so we need to take the blow up $$\widehat{Z}_{R_{\delta}}\rightarrow Z_{R_{\delta}}$$ along $G_{\delta}\cdot Z_{R_{\delta \zeta}}$ to compute $P^{N(R_{\delta})}_{t}(\widehat{Z}_{R_{\delta}}^{ss})$. Note from Proposition \ref{3.3}, we know the blowup $\widehat{Z}_{R_{\delta}}^{ss}/N(R)\rightarrow Z_{R_{\delta}}/N(R)=\mathbb{P}^1$ does not change cohomology, ie, $P_t(\widehat{Z}_{R_{\delta}}^{ss}/N(R))=1+t^2$. Thus $$A_5= \frac{t^2}{1-t^2}(1+t^2) \frac{1}{(1-t^2)(1-t^4)}-\sum  \ mod\ t^{20}$$
where $\sum$ is the terms due to removing the unstable strata of  representation of $R_\delta$ on the normal vector space of some point in $Z_{R_{\delta}}$ .

In order to find the weight of normal representation, we choose a point in  $Z_{\delta}$ distinct to $ \zeta $ whose equation is $F=x_{0}x_{4}x_{5}+f  $, where $f$ is a generic cubic polynomial in $x_1,x_2,x_3$. For normal representation $R_\delta$ on $ \mathcal{N}_{F}$, we take  weight $(2,0,0,0,-1,-1)$  of the maximal torus of  $R_\delta$. Here we view  the weight embedded into the  Lie algebra of $G$. By subtracting the weight from $\frac{\partial F}{\partial x_0},...,\frac{\partial F}{\partial x_5}$ and  the form $x_0q+f$,  we have weight of normal space $\mathcal{N}_{F}$ in following list:
$$\begin{tabular}{|c|c|c|c|c|c|c|}
\hline
  weight& 6& 4 & 0 & -3&-2&-1  \\
    mul& 1 & 3 & 3 & 2&6&6  \\
     \hline
 \end{tabular}$$
Thus, by formula \ref{3.4}, the removing term $\sum$ is $$P_t(\widehat{Z}_{R_{\delta}}^{ss}/N(R))\cdot (t^{2\cdot 8} P_t(\mathbb{P}^{5})+t^{2\cdot 6} P_t(\mathbb{P}^{1})) \ mod\ t^{20}$$ and  the correction term for $5$-th blowup is

\begin{equation*}\label{}
\begin{aligned}
A_5\equiv \frac{t^2}{(1-t^2)^3}-\frac{(t^{12}+...+t^{20})(1+t^2)}{1-t^2}\  \mod\ t^{20}.
\end{aligned}
\end{equation*}

\subsubsection{Computation of  $P^{G}_{t}(X^{ss}_{6})$} \label{4.1.7} We take the blowup $$\pi : \  X_{6}\rightarrow X_{5}^{ss}$$ along $G\cdot \widehat{Z}_{R_{\alpha}}^{ss}$ where $\widehat{Z}_{R_{\alpha}}^{ss}$ is the strict transformation of $Z_{R_{\alpha}}^{ss}$ under previous blowups, since $\alpha$ contains point $\zeta$. It is easy to see 
$$Z_{R_{\alpha}}=\mathbb{P}\{(x_{0}q_{0}(x_{2},...,x_{5})+x_{1}q_{1}(x_{2},...,x_{5})\},$$
$$1+d_{\alpha}=\codim G \widehat{Z}_{R_{\alpha}}=55-(35+19-19)=20,$$
$$N(R_{\alpha})=\{\diag(A,B):\det(A) \cdot \det (B) =1,A\in GL(2,\mathbb{C}),B\in GL(4,\mathbb{C})\}.$$
The blowup $$ \widehat{Z}_{R_{\alpha}}\rightarrow Z_{R_{\alpha}}^{ss}$$ along $N(R_\alpha) \zeta$ descending to quotients will not change cohomology of quotients as in the case of 5-th blowup and thus  we have the formula
$$ P^{N(R_{\alpha})}_{t}(\widehat{Z}_{R_{\alpha}}^{ss})=P_{t}(N(R_{\alpha}))(1+t^2)=\frac{1+t^2}{1-t^2}$$
since $Z_{R_{\alpha}} \q R_\alpha\cong \mathbb{P}^1$ where first identity is due to formula \ref{sps}.

To determine the normal representation of $R_\alpha$, we choose a point in $Z_{R_{\alpha}}^{ss} $ represented by  $F=x_{0}(x_{2}^{2}+x_{3}^{2}+x_{4}^{2})+x_{1}x_{5}^{2}$ which is not in the orbit $\zeta$. Then
$$F_{x_{0}}=x_{2}^{2}+x_{3}^{2}+x_{4}^{2},\ F_{x_{1}}=x_{5}^{2},\ F_{x_{2}}=2x_{0}x_{2}$$
$$F_{x_{3}}=2x_{0}x_{3},\ F_{x_{4}}=2x_{0}x_{2},\ F_{x_{5}}=2x_{1}x_{5}$$ where $F_{x_{i}}:=\frac{\partial F}{\partial x_{i}}$. It is known as before that the tangent space at $F$ is spanned by the monomials $F_{x_{0}},..F_{x_{5}}$ and  monomials in $Z_{R_\alpha}$. Subtracting from $\mathbb{C}[x_{0},x_{1},...,x_{5}]_{3}$, we obtain the normal vector space $$\mathcal{N}_{F}=\mathbb{C}[x_{0},x_{1}]_{3}\oplus span_\bC \{x_{5}V_{2}, x_{1}^{2}x_{2},x_{1}^{2}x_{3},x_{1}^{2}x_{4}\} \oplus V_{3}$$
where $V_2$ is the set of monomials in $x_2,x_3,x_4$ of degree $2$ and $V_3$ is  the vector space of monomials in $x_2,x_3,x_4$ of degree $3$  without the terms $x_2F_{x_{0}}, x_3F_{x_{0}}, x_4F_{x_{0}}$.
Recall the weight of $R_\alpha$ is $(2,2,-1,-1,-1,-1)$, then the weight of normal representation is given by
$$\begin{tabular}{|c|c|c|c|c|}
\hline
   weight &-3& 0 & 3 & 6  \\
   mul & 7 & 5 & 4 & 4  \\
     \hline
 \end{tabular}$$
 So the smallest codimension of unstable strata is $19-8=11$ and thus the removing term vanishes after $mod\ t^{20}$.

In a summary, the correction term in the $6$-th blowup contributes
\begin{equation}
  A_6\equiv \frac{t^2}{1-t^2}\cdot \frac{1+t^2}{1-t^2} \ mod\ t^{20}.
\end{equation}

\subsubsection{Computation of  $P^{G}_{t}(X^{ss}_{7})$}\label{4.1.8}
We take the blowup $$\pi : \  X_{7}\rightarrow X_{6}^{ss}$$ along $G\cdot \widehat{Z}_{R_{\gamma}}^{ss}$ where $\widehat{Z}_{R_{\gamma}}^{ss}$ is the strict transform of $Z_{R_{\gamma}}^{ss}$ under $X_{6}^{ss}\rightarrow X$.
Then the codimension of $GZ_{R_\gamma}$ is given by $$1+d_\gamma=55-(35+14-13)=19.$$
Also we have the normalizer subgroup $$N(R_{\gamma})=\{\diag(a,A,B): a\cdot \det(A) \cdot \det (B) =1,A\in GL(2,\mathbb{C}),B\in GL(3,\mathbb{C})\}.$$
By Proposition \ref{3.3}, we have $Z_{R_{\gamma}}\q G_\gamma\cong \mathbb{P}^1 \times \bC$ and blowup at two points in $Z_{R_{\gamma}}\q  G_\gamma$ will give $P_t(\widetilde{Z_{R_{\gamma}}}\q  G_\gamma)=1+3t^2$.

Now we consider the normal representation $R_{\gamma} \curvearrowright \mathcal{N}_{\gamma}$. As before, by choosing a suitable element in  $\widehat{Z}_{R_{\gamma}}^{ss}$, we compute its normal vector space $\mathcal{N}_{\gamma}$, which is a vector space spanned by monomials in the following form 
\begin{equation*}
    \begin{split}
        &\{\ x_{0}^{3},...,x_{5}^{3}, x_{0}^{2}x_{1},...,x_{0}^{2}x_{4},\\
       & x_{0}x_{1}^{2},x_{0}x_{2}^{2},x_{1}^{2}x_{2},x_{3}^{2}x_{4},x_{3}^{2}x_{5},x_{4}^{2}x_{4},x_{4}^{2}x_{3}\ \}.
    \end{split}
\end{equation*}
It can be identified as the normal representation in  the third blowup in the case of K3 surfaces of degree 2 in \cite{kirwan1989cohomology}, then
in a summary the 7th blowup contributes
\begin{equation}\label{}
\begin{aligned}
A_7\equiv &(t^2+...+t^{2d_\gamma})P_t^{N(R_\gamma})(\widetilde{Z_{R_{\gamma}}})-\sum \hbox{unstable}\ \ \mod\ t^{20}\\
\equiv &\frac{1+3t^2}{1-t^2}(t^2+t^4+...+t^{14}) \ \mod\ t^{20}.
\end{aligned}
\end{equation}
\subsubsection{Computation of $P^{G}_{t}(X^{ss}_{8})$}
We take the blowup $$\pi : \  X_{8}\rightarrow X_{7}^{ss}$$ along $G\cdot \widehat{Z}_{R_{\mu}}^{ss}$ where $\widehat{Z}_{R_{\mu}}^{ss}$ is the strict transform of $Z_{R_{\mu}}^{ss}$ under previous  blowups.  The normalizer subgroup is
\begin{equation*}
    \begin{split}
N(R_{\mu})=\{\diag(a,b,A,c,d):abcd \cdot \det(A) =1,a,b,c,d\in \mathbb{C}^{\ast}, A\in GL(2,\mathbb{C}))\}.
\end{split}
\end{equation*}
The locus $Z_{R_{\mu}}$ is identified as 
\begin{equation*}
    Z_{R_{\mu}}=\mathbb{P}\{ax_{0}x_{4}^{2}+x_{0}x_{5}l_{1}(x_{2},x_{3})+bx_{1}^{2}x_{5}+x_{1}x_{4}l_{2}(x_{2},x_{3})+c(x_{2},x_{3})\}.
\end{equation*}
Here $\{ f\}$ means the vector space spanned by monomials in $f$. Thus the codimension is
\begin{equation*}
    1+d_\mu=\codim GZ_{R_\beta}=55-(35+9-7)=18.
\end{equation*}
The blowup $\widehat{Z}_{R_{\mu}}\rightarrow Z_{R_{\mu}}$ along the orbit $N(R_{\mu})\zeta$ and $N(R_{\mu})\omega$ will decent to the blowup along two points in $Z_{R_{\mu}}\q  N(R_\mu)\cong \mathbb{P}(1,3,6,8)$.  This gives
\begin{equation}\label{}
\begin{aligned}
  P_t^{N(R_{\beta})}(\widehat{Z}_{R_{\beta}})=&P_t(BR_\beta)\cdot P_t(\widehat{Z}_{R_{\beta}}\q  N(R_{\beta}))\\
  =&\frac{1}{1-t^2}\cdot (P_t(\mathbb{P}(1,3,6,8))+(t^2+t^4)+(t^2+t^4)).
  \end{aligned}
\end{equation}
The normal representation for $R_\mu$ can be identified in K3 case as done in last blowup of Kirwan-Lee (see 5.3 in \cite{kirwan1989cohomology}), thus the removing term is giving by
\begin{equation*}\label{}
  \frac{t^{18}}{1-t^2}\ \mod\ t^{20}.
\end{equation*}
In a summary, the correction term will be given by
\begin{equation}\label{}
\begin{split}
     A_8(t)\equiv &\frac{t^2}{1-t^2}\cdot \frac{1+3t^2+3t^4+t^6}{1-t^2} \\
 &- \frac{t^{18}\cdot (1+3t^2+3t^4+t^6)}{1-t^2}\ \mod\ t^{20}.   
\end{split}
\end{equation}
\subsection{Proof of Theorem \ref{thm1.1}}
By previous computations, we have
\begin{equation*}
\begin{aligned}
  P_{t}(\widetilde{\mathcal{M}})=&P^{G}_{t}(X^{ss})+\sum_{i=1}^{8}A_{i}(t) \\
  =&1+9t^2 +26t^{4}+51t^{6}+81t^{8}+115t^{10}+152t^{12}\\
  &+193t^{14}+236t^{16}+280t^{18}+324t^{20} \  \mod\  t^{20}.
  \end{aligned}
\end{equation*}
Then the duality of intersection cohomology will imply the formula.

\section{Intersection cohomology of Baily-Borel compactification} \label{se5}
 In this section, we will compute the intersection cohomology of Baily-Borel compactification $\overline{\mathcal{D}/ \Gamma}^{BB}$ based on the computations in previous sections.
 
\subsection{Baily-Borel compactification of moduli space of cubic fourfolds}
It is well-known that for a smooth cubic fourfold $X$ its integral middle cohomology $H^4(X,\bZ)$ is isomorphic to $$\Lambda:=<1>^{\oplus 21} \oplus <-1>^{\oplus 3}.$$ Let $h:=c_1(\cO_X(1))^2\in \Lambda$ be the hyperplane class and $\Lambda_0=E_8^2 \oplus U^2\oplus A_2=h^\perp$ be the lattice associated to the smooth cubic fourfold, which is isomorphic to the  primitive cohomology $H_p^4(X,\bZ)$ of $X$.  Denote by
$$\mathcal{D}:=\{z\in \mathbb{P}(\Lambda_0\otimes\mathbb{C}):\langle z,z\rangle=0,\langle z,\overline{z}\rangle >0\}$$ the  peroid domain. It is a symmetric domain of type IV.  Let $\Gamma$ be the monodromy group of cubic fourfolds, then it is shown in \cite{beauville} that $\Gamma = O^\ast(\Lambda_0)$ is the automorphism group of $\Lambda_0$ whose elementd act trivially on the discriminant group of $\Lambda_0$. The quotient space $\mathcal{D}/ \Gamma$ is known as a locally symmetric space. By the general result of Baily-Borel in \cite{baily1966compactification}, there is a compactification of $\mathcal{D}/ \Gamma$, whose boundaries  correspond to Type  $II$, $III$ degenerations of cubic fourfolds (see also \cite{looijenga2003compactifications} for refinements ). Such compactification is well-known as Baily-Borel compactification now and we denote by $\overline{\mathcal{D}/ \Gamma}^{BB}$.  Following Hassett \cite{hassett2000special}, we define 
\begin{defn}
A cubic fourfold $X$ is called a special cubic fourfold of discriminant $d$ if it contains a surface $T$ which is not homologous to a complete intersection and the classes $h$ and $[T]$ form a saturated rank 2 sublattice of $\Lambda$ with  discriminant $d$.
\end{defn} \label{heeg}

\begin{exmp}
Let $X$ be a nodal cubic fourfold. That is, $X$ contains a point $x$ whose projective tangent cone is a smooth quadratic. Assume the coordinate of $x$ is $[0,0,0,0,1]$, then $X$ has defining equation of the following form 
\begin{equation*}
    x_5q(x_0,\cdots, x_4)+ c(x_0,\cdots, x_4)=0.
\end{equation*}
where $q$ and $c$ are quadratic and cubic forms in $x_0,\cdots, x_4$. By associateing a K3 surface of degree $6$, Hassett showed that  a nodal cubic fourfold is a special cubic fourfold of discriminant $6$.
\end{exmp}

\begin{exmp} Let $X$ be a smooth cubic fourfold containing a plane $P$, then it is  not hard to check that the class $[P]\in H^4(X,\bZ)$ with  the product of hyper plane class $h^2$ form the following  matrix 
$$\left.\begin{array}{c|c|c} 
&  h^2 & [P] \\ \hline
h^2& 3  & 1 \\   \hline
[P] & 1  & 3 \\  
\end{array} \right. .$$ Such cubic fourfold is a special cubic fourfold of discriminant $8$.
\end{exmp}

It is also shown by Hassett in \cite{hassett2000special} that the locus of special cubic fourfolds of discriminant $d$ is  nonempty for $d \equiv 0,2 \mod 6$.  Moreover, it is a divisor in the moduli space of cubic fourfolds and is called Hassett divisor now. On $\overline{\mathcal{D}/ \Gamma}^{BB}$, we can use the saturated rank 2 sublattices of $\Lambda$ with  discriminant $d$ to define a divisor wich is called a Heegner divisor of  discriminant $d$.  Let $\HH_\infty$ be such a divisor on  $\overline{\mathcal{D}/ \Gamma}^{BB}$ of discriminant $2$.
\begin{thm}(Global Torelli, see \cite{laza2010moduli} \cite{looijenga2009period}) \label{torelli}
The period map $$p: \overline{\mathcal{M}} \dashrightarrow \overline{\mathcal{D}/ \Gamma}^{BB} $$ is a birational map. It is  an open immersion over $\mathcal{M}^o$ and can be defined over $\mathcal{M}$, whose image is the complement of Heegner divisor $\HH_\infty$.
\end{thm}
Thanks to the Torelli theorem \ref{torelli}, these divisors defined in \ref{heeg} are also called Heegner divisors if we view $\mathcal{D}/ \Gamma$ as a Shimura variety. We refer the readers to \cite{MR2192012} for the definition and properties of Shimura varieties.

\begin{rem}
Recently, the property of  open immersion of period map on $\mathcal{M}^o$ is also proven by Huybrechts  and Rennemo in \cite{MR3904800} by using Jacobian rings.
\end{rem}
\subsection{Intersection cohomology}
Let $\widehat{\mathcal{M}}$ be the blowups of $X\q SL(6)$ only along the point $\omega$ and then the line $\chi$.  Then there is a natural contraction morphism 
\begin{equation*}
   f:\  \widetilde{\mathcal{M}}\  \rightarrow \ \widehat{\mathcal{M}}.
\end{equation*}
From \cite{laza2010moduli}, it is known that the period map from GIT compatification to Baily-Borel compatification is resolved by Loojigenga's semi-toric compatification (see for \cite{looijenga2003compactifications} the general  discussions of Loojigenga's semi-toric compatifications):
\begin{equation}\label{5.2}
 \begin{tikzcd}[column sep=small] & \arrow[dl,"p_1"]  \widehat{\mathcal{M}}  \arrow[dr,"p_2"] & \\ \overline{\mathcal{M}} 
\arrow[rr, dashrightarrow, "p"] & & \overline{\mathcal{D}/ \Gamma}^{BB}
\end{tikzcd}
\end{equation}
where $p_1$ is the  composition of the two successive blowups along  the point $\omega$ and then the line $\chi$. Let $ \overline{\mathcal{D}/ \Gamma}^{\Sigma(\HH_\infty)}$ be the Looijenga's  semi-toric compatification associated to the Heegner divisor $\HH_\infty$. We observe the following explicit description of birational morphism $p_2$
\begin{prop} \label{prop4.5} The morphism $p_2$ is the composition of $f:\widetilde{\mathcal{M}} \rightarrow \overline{\mathcal{D}/ \Gamma}^{\Sigma(\HH_\infty)}$  and $\nu:\overline{\mathcal{D}/ \Gamma}^{\Sigma(\HH_\infty)}\rightarrow \overline{\mathcal{D}/ \Gamma}^{BB}$. Here $f$ is the  morphism  contracting the divisor $E_\chi $ to 
\begin{equation}\label{div}
    E_\chi \cap \widetilde{E}_\omega \cong \bP(H^0(C,\cO_C(4)) \oplus H^0(C,\cO_C(6))) \q SO(3) \subset    \overline{\mathcal{D}/ \Gamma}^{\Sigma(\HH_\infty)}
\end{equation}
where $C$ is a smooth plane conic  and $\nu$ is a small modification whose boundaries maps are 
  described in the following table 3. 
\begin{table}[ht]
\caption{Contraction locus of $p_2$ }\label{tab:2}
\centering
  \begin{tabular}{|c|c|c|}
    \hline
     Exceptional locus in $\overline{\mathcal{D}/ \Gamma}^{\Sigma(\HH_\infty)}$ & Boundaries in $\overline{\mathcal{D}/ \Gamma}^{BB}$  &  fiber \\ \hline
    $\phi_\infty$ &  $A_{17}$ &  $\bP^1$\\ \hline
      $\gamma_\infty$ &  $E_7\oplus D_{10}$ &  $\bP^1$\\ \hline
     $\beta_\infty$   &  $E_8^{\oplus 2}\oplus A_2$ &  $\bP^2$ \\ \hline
     $\epsilon_\infty$    &  $A_2\oplus D_{16}$ &  $\bP^2$ \\ \hline
  \end{tabular}
\end{table} 
Here the locus  $\phi_\infty,\cdots ,\epsilon_\infty$ in table \ref{tab:2} is described  in Lemma 6.9 in \cite{laza2010moduli}.
\end{prop}
\begin{proof}
From Section 6 in  \cite{laza2010moduli}, we know that $p_2$ is the composition of a small modification in the sense of Looijenga (see \cite{looijenga2003compactifications}) and a blowup of codimension 2 self-intersection of Heegner divisor $\HH_{\infty}$. It remains to prove the codimension 2 self-intersections is isomorphic to $E_\chi \cap \widetilde{E}_\omega$. This also follows from geometric description of the resolution in Section 6 in \cite{laza2010moduli}. Indeed,  denote by $\widetilde{\HH_\infty}$ its strict transformation and then $\widetilde{\HH_\infty}$ is isomorphic to the blowup of GIT space of plane sextics at a point from  Corollary \ref{cor}. So another divisor $E_\chi$ obtained from morphism $p_1$ is isomorphic to the exceptional divisor of blowup the codiemsion 2 locus (note that such locus is irreducible by Lemma 6.8 in \cite{laza2010moduli}). As the divisor $E_\chi\rightarrow \bP^1$ is fibration, By the construction in Section 7 in  \cite{looijenga2003compactifications}, we get the center is isomorphic to the general fiber of the fibration $E_\chi \rightarrow \bP^1$,  ie, $E_\chi \cap \widetilde{E}_\omega \cong \bP(H^0(C,\cO_C(4)) \oplus H^0(C,\cO_C(6))) \q SO(3)$ where the identification is  from the first blowup of GIT space of plane sextics obtained by Shah in \cite{shah}.

The small modification is determined by the self-intersections of the Heegner divisor $\HH_\infty$ and  its intersection  with Baily-Borel boundaries:  As $\HH_\infty$ has no 1-dimensional self-intersection by Lemma 6.8 in \cite{laza2010moduli}, we can apply  Proposition 7.2 \cite{looijenga2003compactifications} and thus we know  $\nu:\overline{\mathcal{D}/ \Gamma}^{\Sigma(\HH_\infty)}\rightarrow \overline{\mathcal{D}/ \Gamma}^{BB}$ is a normalised blowup, ie, the blowup along the Baily-Borel boundaries that intersect with the self-intersections of $\HH_\infty$. The root lattices of these boundaries and codimesnion  are described in Lemma 6.9 in \cite{laza2010moduli}. Thus, we get the table \ref{tab:2}.
\end{proof}

\begin{prop}
 Let $f:\ X \rightarrow Y$ be a birational morphism of $n$ dimensional irreducible varieties over $\bC$ contracting a divisor $E$ to a lower dimensional locus $Z$ and the restriction $f_E$ of  $f: E \rightarrow Z$ is a topological $\bP^m$-bundle, then for $k\leq n$, the intersection cohomology has a decomposition 
 \begin{equation} \label{bbdg}
     \IH^k(X)\cong \IH^k(Y)\mathop{\oplus} \limits_{ 2 \leq j\leq 2m} H^{k-c+j}(Z,\bQ)
 \end{equation}
 where $c$ is the codimension of $Z$ in $X$.
\end{prop}
\begin{proof}
Let $IC_X$ be the intersection complex on $X$. By BBDG's decomposition theorem, there is a decomposition (non-canonical in general, see \cite{de2009decomposition} for more general results) 
\begin{equation*}
    Rf_\ast IC_X \cong IC_Y \oplus IC_Z(\aL_j) [-i]
\end{equation*}
where $\aL_j$ are  the local systems on $Z$ and $i$ is the degree to be shifted. Following \cite{grushevsky2017intersection}, we can determine these local system:  each $\aL_j$ is an irreducible summand of $R^jf_{E\ \ast}\bQ_E$ where $f_E$ is the morphism restricting on $E$.  $\aL_j$ is rank $=1$ for $j$ even since each fiber is $\bP^m$, thus 
$\aL_j=R^jf_{E\ \ast}\bQ_E$ and the shift degree is $-i=-j+c$.  then by taking cohomology of the decomposition,  
we obtain the formula \ref{bbdg}:

\begin{equation*}
    \begin{split}
        \IH^k(X)=&H^k(Y,IC_Y)\mathop{\oplus} \limits_{2 \leq j\leq 2m} R^jf_{E\ \ast}\bQ_E[-j+c])\\
        =& \IH^k(Y)\mathop{\oplus} \limits_{ 2 \leq j\leq 2m} H^k(Z,\bQ[-j+c])\\
        =& \IH^k(Y)\mathop{\oplus} \limits_{ 2 \leq j\leq 2m} H^{k-c+j}(Z,\bQ).
    \end{split}
\end{equation*}
\end{proof}
Recall that an  algebraic map $f:\ X \rightarrow Y$ is called  semismall if the defect 
\begin{equation*}
    r(f):=\max \{\ i\in \bZ:\  ^p\HH^i (Rf_\ast IC_X[n]) \neq 0\ \}
\end{equation*}
is zero.
\begin{prop} 
  Let $f:\ X \rightarrow Y$ be a semismall birational morphism of $n$ dimensional irreducible varieties over $\bC$ such that $Z\subset Y$ is a connected closed subvariety and $f$ is isomorphic outside $Z$ and over $Z$, $f$ is a $\bP^m$-bundle, then for $k\leq n$
  \begin{equation}\label{5.7}
      \IH^k(X)=\mathop{\oplus} \limits_{ 0 \leq j\leq m} H^{k+2j-n}(Z,\bQ)
  \end{equation}
  where $H^{l}(Z,\bQ)=0$ if $l<0$.
\end{prop}
\begin{proof}
By semi-small property and semi-simplicity of the decomposition theorem, we have 
\begin{equation}
    \begin{split}
        Rf_\ast IC_X[n]=& \mathop{\oplus} \limits_{-r(f) \leq i \leq r(f)} \ ^p\HH^i (Rf_\ast IC_X[n])[-i]\\
        =& \mathop{\oplus} \limits_{j}\ IC(\overline{Y}_i,\aL_{i,j})
    \end{split}
\end{equation}
where $\aL_{i,j}$ is a local system supported on the 
closure $\overline{Y}_i$ of strata $Y_i$. In our case, there is a natural stratification $Y_0=Z, Y_1=Y-Z $, then we have 
\begin{equation}
    Rf_\ast IC_X=IC_Y\mathop{\oplus} \limits_{j=2}^{m} \ \bQ_Z[n-2j]
\end{equation}
By taking the cohomology, we obtain the result.
\end{proof}

\begin{thm}
The intersection cohomology of $\widehat{\mathcal{M}}$ is given by
\begin{equation}
\begin{split}
    \IP_t(\widehat{\mathcal{M}})=&1+3t^2+8t^4+17t^6+29t^8+44t^{10}+61t^{12}+78t^{14}\\
    &+99t^{16}+121t^{18}+151t^{20}+121t^{22}+99t^{24}+78t^{26}\\
   & +61t^{28}+44t^{30}+29t^{32}+17t^{34}+8t^{36}+3t^{38}+t^{40}
\end{split}
\end{equation}
\end{thm}
\begin{proof}
We will use the blowup formula  \ref{ic} of intersection cohomology reversely. Then  we need to do the calculations step by step:
\begin{enumerate}
    \item Blow down $E_\mu$: in this case, $\pi_0(N_\mu)$ acts on the fiber  trivially since $N_\mu$ is connected, thus we need to shift the polynomial by degree $2$ according to the formula \ref{ic}, then we get 
\begin{equation}
\begin{split}
    B_\mu (t)=&(1+3t^2+3t^4+t^6)\cdot (t^2+t^4+2t^6\\
    &+2t^8+3t^{10}+3t^{12}+4t^{14}+4t^{16}+4t^{18}\\
&+4t^{20}+3t^{24}+2t^{26}+2t^{28}+t^{30}+t^{32}).
\end{split}
\end{equation}
\item  Blowing down $E_\gamma$: it is similar to the case $E_\mu$. we get \begin{equation}
\begin{split}
    B_\gamma (t)=&(1+3t^2+t^4)\cdot (t^2+2t^4+3t^6+4t^8\\
    &+5t^{10}+6t^{12}+7t^{14}+8t^{16}+8t^{18}+8t^{20}+7t^{22}\\
&+6t^{24}+5t^{26}+4t^{28}+3t^{30}+2t^{32}+t^{34}).
\end{split}
\end{equation}
\item Blowing down $E_\alpha$:  it is similar to the case $E_\mu$. we get 
\begin{equation}
\begin{split}
    B_\alpha (t)=&(1+t^2)\cdot (t^2+t^4+2t^6+3t^8+4t^{10}\\ &+5t^{12}+6t^{14}+7t^{16}+8t^{18}+8t^{20}+7t^{22}+6t^{24}\\
    &+5t^{26}+4t^{28}+3t^{30}+2t^{32}++t^{34}+t^{36}).
\end{split}
\end{equation}
\item Blowing down $E_\delta$:  it is similar to the case $E_\mu$. we get 
\begin{equation}
\begin{split}
    B_\delta (t)=& (1+t^2)\cdot (t^2+2t^4+4t^6+6t^8+9t^{10}+12t^{12}+\\
    &16t^{14}+19t^{16}+24t^{18}+24t^{20}+19t^{22}+16t^{24}\\
    &+12t^{26}+9t^{28}+6t^{30}+4t^{32}+2t^{34}+t^{36}).
\end{split}
\end{equation}
\item Blowing down $E_\tau$: it is similar to the case $E_\mu$. we get 
\begin{equation}
\begin{split}
    B_\tau (t)=&(1+t^2)\cdot (t^2+t^4+2t^6+3t^8+4t^{10}\\
   & +5t^{12}+7t^{14}+8t^{16}+9t^{20}+8t^{22}+7t^{24}\\
&+5t^{26}+4t^{28}+3t^{30}+2t^{32}+t^{34}+t^{36}).
\end{split}
\end{equation}
\item Blowing down $E_\xi$: note in this case, $\pi_0(N_\xi)=S_6$ acts on $H^\ast(\bP N_\xi/R_\xi)$ by permutation of coordinates of $\bP N_\xi$,
thus, 
\begin{equation}
   \begin{split}
        \IP_t(H^\ast(\bP N_\xi/R_\xi)^{\pi_0(N_\xi)}& \equiv
        P_t(\bP N_\xi)P_t((H^\ast(BR_\xi)^{\pi_0(N_\xi)})\ \mod\ t^{19}\\
       & \equiv \frac{1}{\mathop{\Pi} \limits_{1\leq i \leq6}(1-t^{2i})}  \ \mod \ t^{19}
   \end{split}
\end{equation}
then using formula \ref{ic} again, we have
\begin{equation}
\begin{split}
    B_\xi (t)=& t^2+t^4+2t^6+3t^8+5t^{10}+7t^{12}+11t^{14}+14t^{16}\\
    &+20t^{18}+26t^{20}+20t^{22}+14t^{24}+11t^{26}\\
&+7t^{28}+5t^{30}+3t^{32}+2t^{34}+t^{36}+t^{38}
\end{split}
\end{equation}
\end{enumerate}

Put these together, we obtain our formula from $$\IP_t(\widehat{\mathcal{M}})=P_t(\cM)- B_\mu (t)-B_\xi (t)-B_\delta (t).$$
\end{proof}

\begin{rem}
 In \cite{casalaina2019cohomology}, the authors doubted whether the Kirwan resolution of moduli spaces of cubic threefolds is isomorphic to certain toroidal compactification of ball quotient $\bB/ \Gamma$ with respect to  some cone decomposition. Their evidence in \cite{casalaina2019cohomology} is that they compute the cohomology of the toroidal compactification  and find the Betti numbers of the two compactification  match perfectly. It is quite interesting to  ask whether it is also true for the moduli spaces of cubic fourfolds. 
\end{rem}

\begin{cor}
The intersection Betti numbers of $\overline{\mathcal{D}/ \Gamma}^{BB}$ are given by
\begin{equation} \label{cor4.10}
    \begin{split}
    \IP_t(\overline{\mathcal{D}/ \Gamma}^{BB})=&1+2t^2+5t^4+13t^6+24t^8+38t^{10}+54t^{12}+70t^{14}\\
    &+88t^{16}+107t^{18}+137t^{20}+107t^{22}+88t^{24}+70t^{26}\\
    &+54t^{28}+38t^{30}+24t^{32}+13t^{34}+4t^{36}+2t^{38}+t^{40}.
    \end{split}
\end{equation}

\end{cor}
\begin{proof}
First by applying formula \eqref{bbdg} to morphism $f$ in  Proposition  \ref{prop4.5}, we have 
\begin{equation*}
    \IP_t(\overline{\mathcal{D}/ \Gamma}^{\Sigma(\HH_\infty)})=\IP_t(\widehat{\mathcal{M}})-(1+t^2)(P_t(E_\chi \cap \widetilde{E}_\omega)-1) \  \mod\ t^{20}.
\end{equation*}
Then it remains to compute the cohomology $E_\chi \cap \widetilde{E}_\omega$. Thanks to \ref{cor}, we identify $E_\chi \cap \widetilde{E}_\omega$ as the exceptional divisor in the  1st blow up of  GIT moduli of degree 2 K3 surfaces.  According to Section 4.2 \cite{kirwan1989cohomology}, after blowup a point  $\chi \cap E_\omega$,   we get a partial resolution of $E_\omega$ and the exceptional divisor $ E_\chi \cap \widetilde{E}_\omega \cong \bP(H^0(C,\cO_C(4)) \oplus H^0(C,\cO_C(6))) \q SO(3)$ by Proposition \ref{prop4.5}. Note that $\bP(H^0(C,\cO_C(4)) \oplus H^0(C,\cO_C(6))) \q SO(3)$ has only quotient singularities at worst. As the smooth plane conic $C$ is isomorphic to $\bP^1$, this will induces isomorphism $$\bP(H^0(\cO_C(4)) \oplus H^0(\cO_C(6))) \q SO(3)\cong \bP(H^0(\cO_{\bP^1}(8)) \oplus H^0(\cO_{\bP^1}(12))) \q SL(2).$$ Let $T=\{ \diag(t,t^{-1}):\ t\in \bC^\ast\}$ be a maximal torus of $SL(2)$ and $$[a_0,\cdots,a_8,b_0,\cdots,b_{12}]$$ be the homogeneous coordinate for $\bP (H^0(\bP^1,\cO(8)) \oplus H^0(\bP^1,\cO(12))$, then the action  $T$  on $\bP (H^0(\bP^1,\cO(8)) \oplus H^0(\bP^1,\cO(12))$ is given by  $$t\cdot [a_0,\cdots,a_8,b_0,\cdots,b_{12}]=[t^{-8}a_0,t^{-6}a_1,\cdots,t^{8}a_8,t^{-12}b_0,t^{-10}b_1,\cdots,t^{12}a_{12},].$$ so the number of weight $<0$ is $4+6=10$ and the maximal dimenion of unstable strata for action $SL(2)$ on $\bP (H^0(\bP^1,\cO(8)) \oplus H^0(\bP^1,\cO(12))$ is $9$ which implies 
\begin{equation*}
    \begin{split}
         P_t(E_\chi \cap \widetilde{E}_\omega)=&P_t^{SL(2)}(\bP(H^0(\bP^1,\cO(8)) \oplus H^0(\bP^1,\cO(12))) \q SL(2)\\
         \equiv& P_t(B SL(2))- \hbox{unstable terms}\mod t^{18}\\
         \equiv &  P_t(B SL(2)) P_t(\bP^{21}) \equiv \frac{1}{(1-t^2)(1-t^4)}\ \mod t^{18}.
    \end{split}
\end{equation*}
Thus we have 
\begin{equation*}
   \begin{split}
        P_t(E_\chi \cap \widetilde{E}_\omega)=&1+t^2+2t^4+2t^6+3t^8+3t^{10}+4t^{12} \\
    &+4t^{14}+5t^{16}+5t^{18}+5t^{20}+4t^{22}+4t^{24}\\
&+3t^{26}+3t^{28}+2t^{30}+2t^{32}+t^{34}+t^{36}. 
   \end{split}
\end{equation*}
Last by applying formula \ref{5.7} to morphism $\nu$ and combine the table \ref{tab:2}, we only need to remove  
$$2(t^{18}+t^{20})+2(t^{16}+t^{18}+t^{20})\  \mod\ t^{20}$$ from $\IP_t(\overline{\mathcal{D}/ \Gamma}^{\Sigma(\HH_\infty)})\ \mod\ t^{20}$ in order to get $ \IP_t(\overline{\mathcal{D}/ \Gamma}^{BB})$. In this way, we obtain our formula \ref{cor4.10}.
\end{proof}

\begin{rem}
Since the Zucker's conjecture  was established in \cite{looijenga19882} and \cite{saper1990l2}, the $L^2$-cohomology of $\mathcal{D}/ \Gamma$ is isomorphic to the  intersection  cohomology of $\mathcal{D}/ \Gamma$.  our result also provides most  $L^2$-betti numbers of $\mathcal{D}/ \Gamma$ as the dimension of Baily-Borel's boundaries is $1$.
\end{rem}
\vspace{0.2cm}
\textbf{Acknowledgement:} The author would like to thank his advisor Prof. Meng Chen and Prof. Zhiyuan Li for their constant encouragements and supports, especially many helpful conversations from  Zhiyuan Li. He also thanks Haitao Zou for reading the drafts.


\bibliographystyle {plain}
\bibliography{main}
\end{document}